\documentclass[A4paper, oneside,11pt]{article}

\usepackage{amsmath} 
\usepackage{amssymb} 
\usepackage{amscd} 
\usepackage{tabularx} 
\usepackage[all,cmtip,line]{xy}
\usepackage{enumerate}
\usepackage{hyperref}
\usepackage{bm}

\usepackage{kpfonts}
\usepackage[T1]{fontenc}

\usepackage{apptools}
\AtAppendix{\counterwithin{thm}{section}}

\usepackage{xcolor}
\usepackage{tikz-cd}
\usepackage{amsthm}
\usepackage{tcolorbox}
\theoremstyle{plain}
\newtheorem{thm}{\textbf{Theorem}}[section]

\newtheorem{prop}[thm]{\textbf{Proposition}}
\newtheorem{lem}[thm]{\textbf{Lemma}}
\newtheorem{corollary}[thm]{\text{Corollary}}
\newtheorem{rem}[thm]{\textbf{Remark}}

\theoremstyle{definition}
\newtheorem{defn}[thm]{\textbf{Definition}}
\newtheorem{exer}[thm]{\textbf{Exercise}}

\newtheorem{ex}[thm]{\textbf{Example}}

\usepackage{etoolbox}
\makeatletter
\patchcmd\@thm
  {\let\thm@indent\indent}{\let\thm@indent\noindent}%
  {}{}
\makeatother
\makeatletter
\renewenvironment{proof}[1][\proofname]{\par
  \pushQED{\qed}%
  \normalfont \topsep6\p@\@plus6\p@\relax
  \trivlist
  \itemindent\z@ % original has \normalparindent
  \item[\hskip\labelsep
        \scshape
    #1\@addpunct{.}]\ignorespaces
}{%
  \popQED\endtrivlist\@endpefalse
}
\makeatother

\def\Q{\mathbb{Q}}

\def\bZ{\mathbb{Z}}
\def\Z{\mathbb{Z}}

\def\A{\mathbb{A}}
\def\bA{\mathbb{A}}
\def\R{\mathbb{R}}

\def\ra{\rightarrow}
\def\Max{\mathrm{Max}}

\def\sub{\subseteq}
\def\Acirc{A^{\circ}}
\def\ok{\mathcal{O}_K}

\newcommand{\mc}{\mathcal}

\def\cO{\mathcal{O}}

\def\spa{\mathrm{Spa}}
\def\Spa{\mathrm{Spa}}

\def\T{\mathbf{T}}
\def\cZ{\mathcal{Z}}
\def\cE{\mathcal{E}}

\newcommand{\End}{\operatorname{End}} 
\newcommand{\Hom}{\operatorname{Hom}}
\newcommand{\Spec}{\operatorname{Spec}} 
\newcommand{\Ker}{\operatorname{Ker}}
\newcommand{\mf}{\mathfrak}
\newcommand{\ol}{\overline}
\def\vp{\varpi}
\def\mc{\mathcal}
\newcommand{\ms}{\mathscr}
\definecolor{tree}{HTML}{50925a}

\usepackage[headsep=25pt,textwidth=14cm]{geometry}

\usepackage{fancyhdr}
\usepackage{todonotes}

\author{Judith Ludwig}
\begin{document}
\begin{center}
{\LARGE \textbf{ \textsc{Spectral theory and the Eigenvariety machine}}}\\
\vspace{1cm}
{\Large Notes for a mini-course at the Spring School on}\\[5pt]
{\Large \textit{Non-archimedean geometry and Eigenvarieties}}\\[5pt]
{\Large Heidelberg, March 6$^{th}$-17$^{th}$, 2023}\\[5pt]
\vspace{1cm}
{\large \textsc{Judith Ludwig}}\\[5pt]
\textsc{\date}\\
\end{center}
\thispagestyle{empty}
\vspace{8cm}

\begin{abstract}
These are extended lecture notes for a mini course at the Spring School on Non-Archimedean Geometry and Eigenvarieties held at  Heidelberg University in March 2023. The goal of the course is to explain a modern take on the eigenvariety machine in the language of adic spaces. For this we build  on the theory developed in the first week of the school. We also explain some basic selective non-archimedean functional analysis. 
\end{abstract}

\newpage
\section{Introduction}

The goal of these notes is to explain the eigenvariety machine, an abstract framework in which many eigenvariety constructions fit. 

Eigenvarieties are adic spaces that $p$-adically interpolate systems of Hecke eigenvalues of (certain) automorphic forms. At the same time an eigenvariety serves as a parameter space for systems of Hecke eigenvalues attached to $p$-adic automorphic forms. The first and probably still most important example of an eigenvariety is the so called Coleman--Mazur eigencurve. It interpolates modular forms $p$-adically and has been an important object for number theory and the $p$-adic Langlands program. It was first constructed by Coleman and Mazur in \cite{cm} for modular forms of tame level one\footnote{i.e. for modular forms of level dividing $p$} as a rigid space over $\mathbb{Q}_p$. Buzzard then axiomatized the construction in \cite{buz} and built the eigenvariety machine. This machine is constructed independently of any theory of $p$-adic automorphic forms. As we will see below it requires some functional analysis and some geometry. In practice one then needs to feed the machine with spaces of $p$-adic forms. E.g. in \cite{buz}, Buzzard applied his machine to construct the Coleman-Mazur eigencurve for modular forms of arbitrary tame level and to construct eigenvarieties for Hilbert modular forms. Since then, the eigenvariety machine (or parts of it) have been used in the construction of many more eigenvarieties. We refer to the course of Newton \cite{newton} for details and in particular to Section 3 for an overview of eigenvarieties.  

The eigenvariety machine was originally developed in the language of rigid analytic geometry. Here we present the eigenvariety machine in the language of adic spaces. One reason for this is the following: 
The Coleman--Mazur eigencurve lives over a base called \emph{weight space} which is a finite disjoint union of copies of the rigid analytic generic fibre of $\mathrm{Spf}(\Z_p\llbracket T\rrbracket)$.
In the adic language, this is 
\[\mathcal{W}^{rig}:=\spa(\Z_p\llbracket T\rrbracket, \Z_p\llbracket T\rrbracket)\times_{\spa(\Z_p,\Z_p)} \spa(\Q_p,\Z_p),\] which is an open unit disc. %(compare \ref{opendisc}).
By adding one point  $x_{\mathbb{F}_p\llbracket T\rrbracket}$ (the $T$-adic valuation on $\mathbb{F}_p\llbracket T\rrbracket$) to this space, we get a quasi-compact space which turns out to be the analytic locus $ \mathcal{W}:=\spa(\Z_p\llbracket T\rrbracket, \Z_p\llbracket T\rrbracket)^{an}$ of $\spa(\Z_p\llbracket T\rrbracket, \Z_p\llbracket T\rrbracket)$. Coleman conjectured that the eigencurve can be extended from its original construction over $\mc{W}^{rig}$ into characteristic $p$ and in \cite{aip}, Andreatta, Iovita and Pilloni show that this is indeed possible, i.e., that there is an \emph{adic Coleman-Mazur eigencurve} over $\mathcal{W}$ that extends the construction over $\mathcal{W}^{rig}$. The eigenvariety machine carries over to a more general adic setting and gives the freedom of working over Tate rings that do not live over a field. When presenting the eigenvariety machine in this generality, we closely follow \cite{buz} using the proofs in \cite{aip}, \cite{JN1} and \cite{JN2} for many of the statements in the adic world.
\smallskip
\medskip 

We view the construction of eigenvarieties as a 3-step process.
\medskip

\noindent
\begin{minipage}{0.65\textwidth}
\textbf{Step 1.} Construct \emph{input data} for the machine, i.e., \emph{a base} $\mc{W}$ which is usually referred to as weight space (which interpolates weights of automorphic forms) and \emph{spectral data} locally over the base. In particular this involves the construction of suitable spaces of $p$-adic automorphic forms. For the construction of concrete input data we refer to the notes of James Newton.
\smallskip

\textbf{Step 2.} From the input data, construct the so called spectral variety $\mc{Z}$ over the base $\mc{W}$ which parametrizes non-zero eigenvalues of one special Hecke operator. 
\smallskip

\textbf{Step 3.} Construct the eigenvariety $\mc{E}$, a space which is finite over $\mc{Z}$ and which parametrizes certain systems of eigenvalues appearing in the input datum. 
\end{minipage}
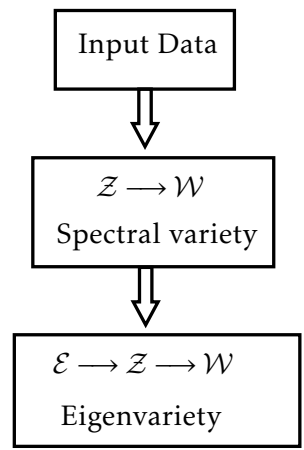
\begin{minipage}{0.03\textwidth}
\phantom{-}
\end{minipage}
\begin{minipage}{0.3\textwidth}
\tikzset{every picture/.style={line width=0.75pt}} %set default line width to 0.75pt        

\tikzset{every picture/.style={line width=0.75pt}} %set default line width to 0.75pt        

\begin{tikzpicture}[x=0.75pt,y=0.75pt,yscale=-1,xscale=1]
%uncomment if require: \path (0,300); %set diagram left start at 0, and has height of 300

%Shape: Rectangle [id:dp29446824738095967] 
\draw  [line width=1.25]  (50.67,0) -- (141,0) -- (141,39.67) -- (50.67,39.67) -- cycle ;
%Down Arrow [id:dp7288754237868463] 
\draw  [line width=1.25]  (89.67,57.07) -- (92.79,57.07) -- (92.79,39.67) -- (99.04,39.67) -- (99.04,57.07) -- (102.17,57.07) -- (95.92,68.67) -- cycle ;
%Shape: Rectangle [id:dp16786768113281714] 
\draw  [line width=1.25]  (39.33,74.33) -- (160.83,74.33) -- (160.83,129.33) -- (39.33,129.33) -- cycle ;
%Shape: Rectangle [id:dp053947080349740384] 
\draw  [line width=1.25]  (30,163.67) -- (174,163.67) -- (174,220) -- (30,220) -- cycle ;
%Down Arrow [id:dp0030164270989569975] 
\draw  [line width=1.25]  (90,147.07) -- (93.13,147.07) -- (93.13,129.67) -- (99.38,129.67) -- (99.38,147.07) -- (102.5,147.07) -- (96.25,158.67) -- cycle ;

% Text Node
\draw (60.67,10.33) node [anchor=north west][inner sep=0.75pt]   [align=left] {Input Data};
% Text Node
\draw (70,82.33) node [anchor=north west][inner sep=0.75pt]   [align=left] {$\mc{Z}\longrightarrow \mc{W}$};
% Text Node
\draw (52,198) node [anchor=north west][inner sep=0.75pt]   [align=left] {Eigenvariety };
% Text Node
\draw (50,105.33) node [anchor=north west][inner sep=0.75pt]   [align=left] {Spectral variety};
% Text Node
\draw (48,171) node [anchor=north west][inner sep=0.75pt]   [align=left] {$\mc{E}\longrightarrow \mc{Z}\longrightarrow \mc{W}$};
\end{tikzpicture}
\end{minipage}
\bigskip

This text is meant for non-experts as an introduction to the eigenvariety machine and exercises are included throughout to help you understand the material. 
There are some excellent references (such as \cite{buz},  \cite{aip}, \cite{JN1}, \cite{JN2} and \cite{bellaiche}) which we often refer to for further details, or when we omit some proofs. None of the following is original work and all mistakes are mine. 
\medskip 

\textbf{Structure of the notes}: In order to construct eigenvarieties, we first have to study some functional analysis. Section 2 takes care of this. Section 2 is also "geometry-free", meaning you can understand it without knowledge of adic spaces. The same is true for Section 4, where we treat Riesz theory, another important piece of functional analysis. Section 3 requires some knowledge on adic spaces. There we construct the so called spectral variety, an important space that sits between the eigenvariety and weight space. In Section 5 we then explain how to build eigenvarieties, first over an affinoid base, which is easier. In a second part we show how the construction glues. Finally we describe some features, when the base space is reasonably well behaved, e.g.\ a pseudorigid space. We also discuss some rigidity results. 
\medskip 

\textbf{Notation}: 
By a non-archimedean field $K$ we mean a topological field $K$ whose topology is induced by a  valuation of rank $1$ and which is complete with respect to this valuation.
All rings in these notes are assumed to be non-zero. We freely use the language of adic spaces and concepts like Tate rings, Huber rings, etc., from the first week of the school.
We follow the same conventions as \cite[Introduction]{johansson} and in particular use the terms strongly noetherian Tate rings and locally noetherian adic spaces as recalled in loc. cit. Unless stated otherwise we always assume the occurring Tate rings are complete. 
\bigskip

\textbf{Acknowledgments}: 
I would like to thank the participants of the spring school for their feedback. In particular I want to thank Jakob Burgi and Alireza Shavali for their helpful comments and suggestions on an earlier draft. I would also like to thank Christian Johansson for help with Proposition \ref{constslope}, him and James Newton for many helpful conversations and the referee for their careful reading and their comments. 

The author acknowledges support from the Deutsche Forschungsgemeinschaft (DFG, German Research Foundation) through TRR 326 \textit{Geometry and Arithmetic of Uniformized Structures},  project number 444845124.

\newpage
\section{Background from Functional Analysis}
This section offers the background from functional analysis that we need for the construction of eigenvarieties (in the language of adic spaces). The central concepts are Banach--Tate rings, (orthonormalizable) modules over them and the theory of compact operators on these modules. These concepts provide the technical framework to handle the input data for the eigenvariety machine. The main references are \cite{buz, JN1} and~\cite{bellaiche}.   
\subsection{Modules over Banach--Tate rings}
\begin{defn} 
Let $A$ be a ring. A \emph{seminorm} on $A$ is a map $|\cdot|\colon A \rightarrow \mathbb{R}_{\geq 0}$ that satisfies: For all $ a,b\in A$,
\begin{enumerate}
\item $|0|=0$, $|1|=1$,
\item $|a+b|\leq {\max}(|a|,|b|)$,  \hfill \emph{(the ultrametric triangle inequality)}
\item $ |ab|\leq |a||b|$. \hfill \emph{($|\cdot|$ is submultiplicative)}
\end{enumerate}
A \emph{norm} on $A$ is a seminorm that satisfies $|a|=0$ only if $a=0$. A ring $A$ together with a (semi)norm is called a (semi)normed ring. A \emph{Banach ring} is a normed ring $A$ which is complete with respect to this norm. 
\end{defn}

\begin{rem} \begin{itemize}
\item Note that in contrast to a valuation (\cite[Def.2.2]{bergdall}), a (semi)norm is only required to be submultiplicative. A multiplicative (semi)norm is an example of a valuation. 
\item By definition the norm is part of the data of a Banach ring (and it will be part of more definitions below). We remark however that ultimately many of the constructions (see e.g. Proposition \ref{ind}), and in particular the eigenvariety construction will only depend on the topological rings involved, and not on the choice of norm that defines the topology. 
\end{itemize}
\end{rem}

\begin{defn} 
Let $A$ be a normed ring. A non-zero element $a\in A$ is called \emph{multiplicative} if $|ab|=|a||b|$ for all $b\in A$.
\end{defn} 

\begin{defn} A \emph{Banach--Tate ring} is a complete normed ring $A$ which contains a unit $\varpi$ such that $|\varpi|<1$ and which is multiplicative. 
\end{defn}

\begin{ex} Let $K$ be a non-archimedean field and $n\in \mathbb{N}$. Then the classical Tate-algebra $\mathcal{T}_n:= K\langle X_1,\ldots,X_n \rangle $ in $n$ variables equipped with the Gauß norm (which is even multiplicative) is a Banach--Tate ring. (See \cite[Chapter 5]{BGR} for details and an extensive discussion of this algebra and its properties).
\end{ex}

Let us relate Banach--Tate rings to the rings underlying (analytic) adic spaces and thereby justify the naming. 
For that, let $A$ be a complete Tate ring with ring of definition $A_0$ and pseudo-uniformizer $\varpi \in A_0$. Let $p$ be a prime number\footnote{Note that more generally (\ref{standnorm}) defines a norm for any positive real number instead of $p$.}. Then we can define a norm on $A$ by setting, for any $a\in A$,
\begin{equation} \label{standnorm}
|a|:= \inf \{p^{-n} \ | \ a \in \varpi^nA_0, n\in \mathbb{Z}\}.
\end{equation}
\begin{exer} Check that this defines a norm on $A$, that $\varpi$ is multiplicative and that the topology induced by this norm agrees with the topology of $A$ as a Tate ring.  
\end{exer}
In the following we use the terminology normed Tate ring to mean a Tate ring, whose topology is also given by a norm (in general unspecified). For a complete Tate ring, we refer to the norm of equation (\ref{standnorm}) attached to a chosen pseudo-uni-formizer as a standard norm. 

\begin{rem} Let $A$ be a normed ring with a unit $\varpi$ such that $|\varpi|<1$ and which is multiplicative. Then the underlying topological ring is a Tate ring. More precisely, the unit ball $\{a \in A : |a|\leq 1\}$ is a ring of definition and $\varpi$ is a pseudo-uniformizer. Conversely, by the exercise above, we see that any complete Tate ring gives rise to a normed ring with a unit $\varpi$ such that $|\varpi|<1$ and which is multiplicative. So we may think of a Banach--Tate ring as a complete normed Tate ring with a chosen multiplicative pseudo-uniformizer. 
\end{rem} 

When treating functional analysis in the context of rigid analytic geometry, one works with Banach rings that contain a non-archimedean field, like the rings $\mathcal{T}_n$ from above. Our setup here is more general. Here is an example of a Banach--Tate ring which does not live over a field.  

\begin{ex} \label{tatenofield} 
Consider the Huber ring $\bZ_p\llbracket T\rrbracket$, with the $(p,T)$-adic topology.
On 
\[X:=\spa(\bZ_p\llbracket T\rrbracket,\bZ_p\llbracket T\rrbracket)\] 
consider the rational subset $U= X\left(\frac{p,T}{T} \right)=\{x\in X : |p(x)|\leq |T(x)|\neq 0 \}$\footnote{See \cite[Section 4.2]{berk} for a detailed description of $X$.}. Then $\cO_X(U)$ is given as follows. Take the ring $R:=\bZ_p\llbracket T\rrbracket [1/T] $. It has the structure of a Huber ring with ring of definition  $ \bZ_p\llbracket T\rrbracket [p/T]$ and the $(T)$-adic topology on the latter. Then $\cO_X(U)$ is equal to the completion $\widehat{R}$ of $R$. It is a Tate ring ($T$ is a topologically nilpotent unit), which does not live over a field ($p$ is neither invertible nor zero in  $\cO_X(U)$). Fixing $T$ turns $\cO_X(U)$ into a Banach--Tate ring with norm as in equation (\ref{standnorm}). 
\end{ex}
\bigskip

\begin{defn}
Let $A$ be a normed ring. A \emph{normed $A$-module} is an $A$-module $M$ equipped with a function $||\cdot|| \colon M \rightarrow \mathbb{R}_{\geq 0}$ such that for all $m,n\in M$ and all $a\in A$, 
\begin{enumerate}
\item $||m||=0 \iff m=0$,
\item $||m+n|| \leq \max(||m||,||n||),$
\item $||am|| \leq |a|\cdot ||m||.$
\end{enumerate}
If $A$ is a Banach ring and $M$ is complete, we say that $M$ is a \emph{Banach $A$-module}. Finally a \emph{Banach $A$-algebra} $B$ is defined to be a Banach ring $B$ which is also a Banach $A$-module. 
\end{defn}

%\begin{ex} Let $K$ be a non-archimedean field and $n\in \mathbb{N}$. Then the classical Tate-algebra $K\langle T_1,\ldots,T_n \rangle $ in $n$ variables equipped with the Gaußnorm is a Banach $K$-module. As the Gaußnorm is sub-multiplicative (even multiplicative), it is also a Banach $K$-algebra. (See \cite[Chapter 5]{BGR} for details and an extensive discussion of this algebra and its properties).
%\end{ex}

\begin{exer} 
\begin{enumerate}
\item Let $A$ be a normed ring and let $\varpi \in A$ be a multiplicative unit. Show that if $M$ is a normed $A$-module, then $||\varpi m||=|\varpi| \cdot ||m||$ for all $m\in M$. 
\item If $M$ and $N$ are Banach $A$-modules then $M\oplus N$ becomes a Banach $A$-module via $|m +n|:= \max \{|m|, |n|\}.$ Iterating this gives a natural Banach $A$-module structure on $A^r$ for any $r\in \mathbb{N}$. 
\end{enumerate}
\end{exer}

\begin{defn} %\begin{enumerate}
%\item If $f\colon A \rightarrow B$ is a morphism of normed rings, we say that $f$ is \emph{bounded} if there is a constant $C > 0$ such that $|f(a)| \le C|a|$ for all $a\in A$. 
%\item 
Let $A$ be a normed Tate ring. An $A$-linear map $\phi\colon M\rightarrow N$ of normed $A$-modules is called \emph{bounded} if there exists $C\in \mathbb{R}_{>0}$ such that $||\phi(m)|| \leq C ||m||$ for all $m\in M$.
%\end{enumerate}
\end{defn}

\begin{lem} \label{contbound} Let $A$ be a normed Tate ring with a multiplicative pseudo-uniformizer and let $M$ and $N$ be normed $A$-modules.
An $A$-linear map $\phi\colon M\rightarrow N$ is continuous if and only if it is bounded.
\end{lem}
\begin{proof} Exercise. (Hint: Adapt the proof of \cite[Appendix B, Lemma 1]{bosch}).\footnote{For a discussion of the relationship of continuity and boundedness in a more general context see \cite[Section 2.1.8]{BGR}.}
\end{proof}

\begin{defn} \begin{enumerate}
\item If $A$ is a normed Tate ring and $M$ and $N$ are normed $A$-modules then a homomorphism $\phi\colon M\rightarrow N$ is defined to be a continuous $A$-linear map. 
\item The norm of a homomorphism $\phi\colon M\rightarrow N$ of normed $A$-modules is defined as
\[||\phi|| = \sup_{m\neq 0} \frac{||\phi(m)||}{||m||}.\]
\end{enumerate}
\end{defn}

If $A$ is a Banach--Tate ring and $M$ and $N$ are Banach $A$-modules, then this norm turns the $A$-module $\Hom_A(M,N)$ of continuous (equiv.\ bounded) $A$-linear homomorphisms from $M$ to $N$ into a Banach $A$-module. Completeness follows for if $\phi_n$ is a Cauchy sequence in $\Hom_A(M,N)$ then for all $m\in M$,  $\phi_n(m)$ is a Cauchy sequence in $N$ so has a limit which we call $\phi(m)$. Then $\phi$ is the limit of the $\phi_n$. 

We will frequently make use of the \emph{open mapping theorem}, which says that for a Banach--Tate ring $A$ and Banach $A$-modules $M$ and $N$, any surjective continuous $A$-linear map $\phi\colon M\rightarrow N$ is open (see \cite[Theorem II.4.1.1.]{mor}).
The following is a useful application. By a finite Banach $A$-module we mean a Banach $A$-module that is finitely generated as an abstract $A$-module.
\begin{lem}\label{autcont} If $M$ is a Banach $A$-module and $P$ is a finite Banach $A$-module, then any abstract $A$-module homomorphism $\phi\colon P\rightarrow M$ is continuous.
\end{lem}
\begin{proof} Let $\pi\colon A^r\rightarrow P$ be a surjection of $A$-modules and equip $A^r$ with its usual Banach $A$-module norm. Note that $\pi$ is bounded hence continuous. By the open mapping theorem, $\pi$ is open. Furthermore $\phi\circ \pi$ is bounded and therefore also continuous. Hence $\phi$ is also continuous. 
\end{proof} 

Our goal now is to develop some basic non-archimedean functional analysis for Banach modules over Banach--Tate rings (eventually under a noetherian assumption). In particular, we study compact operators on them and then (in Section \ref{sec:riesz}) study the spectral theory of these operators. 
When the Banach--Tate ring is a Banach $K$-algebra over a non-archimedean field $K$, then the theory as we present it is explained in great detail in \cite{buz}, building on work of Coleman and Serre \cite{serre}. Further helpful references are \cite[Chapter 3]{bellaiche} and \cite[Section 2]{JN1}. % which also treats the case of Banach $K$-algebras, but there the noetherian assumption is relaxed. 
Most results go through almost word by word in the more general setting of Banach--Tate rings. 
One word of caution: When adapting proofs from the setting of Banach $K$-algebras to Banach--Tate rings, one has to be a bit careful when it comes to "changing the norm". More precisely, we want to show that certain results only depend on the topology of the Banach--Tate ring and do not dependent on the chosen norm. 
In this context let us recall that two norms $|\cdot|$, $|\cdot|^{\prime}$ on a ring $A$ are called \emph{equivalent} 
if they induce the same topology and that they are called
\emph{bounded-equivalent} if there are constants $C_{1},C_{2}>0$ such that 
$C_{1}|a|\leq |a|^{\prime}\leq C_{2}|a|$ for all $a\in A$. 
Now for Banach-algebras over a non-archimedean field, the notions ``equivalent norms'' and ``bounded-equivalent norms'' coincide. For Banach--Tate rings however the situation is different. If two norms on a Banach--Tate ring $A$ are bounded equivalent then they induce the same topology, but the converse might not be true (cf. \cite[Lemma 2.1.6 and 2.1.7]{JN1} for what can be said in the converse direction).
In these notes this issue will only appear in Prop. \ref{ind}. 
Let us also observe, that when the norm on $A$ is fixed and we only change the norm on the Banach-module $M$, life is easier as the following remark shows. 

 \begin{rem} \label{eqnorms} Let $(A,|\cdot|)$ be a Banach--Tate ring and let $M$ be a Banach $A$-module, equipped with two equivalent norms $||\cdot||$ and $||\cdot||'$, i.e., two norms that induce the same topology on $M$ or equivalently such that the identity maps $\mathrm{id} \colon (M,||\cdot||) \rightarrow (M, ||\cdot||')$  and $\mathrm{id} \colon (M,||\cdot||') \rightarrow (M, ||\cdot||)$ are continuous. By Lemma  \ref{contbound} above, we see that two norms on $M$ are equivalent if and only if there are positive constants $C$ and $C'$, such that
 \[C||m||'\leq ||m|| \leq C'||m||', \ \text{ for all } m \in M.\]
 Let $\phi\colon M\rightarrow M$ be a continuous $A$-linear map, then for all $m\in M$ we have 
 \[C||\phi(m)||'\leq ||\phi(m)|| \leq ||\phi||\cdot ||m|| \leq C' ||\phi|| \cdot ||m||',\]
 which implies that $||\phi||' \leq (C'/C) || \phi||$. Similarly we get that $||\phi|| \leq (C'/C) ||\phi||'$, in particular the operator norms $||\cdot|| $ and $||\cdot||'$ on $\Hom_A(M,M)$ are also equivalent.
 \end{rem}

\subsection{Orthonormalizable Banach modules} \label{sec:onbs}
\textbf{Convention}: If $I$ is a set and we are given a map $I\rightarrow A, i\mapsto a_i$ into a Banach--Tate ring $A$, then the statement $\lim_{i\rightarrow \infty} a_i=0$ means that for any $\epsilon >0$ there are only finitely many $i\in I$ with $|a_i|>\epsilon$. 
\begin{exer} Show that $\lim_{i\rightarrow \infty} a_i=0$ implies that only countably many $a_i$ can be non-zero. 
\end{exer}
\begin{exer}
Let $A$ be a %noetherian 
Banach--Tate ring and $M$ a Banach $A$-module. Consider a subset $\{e_i: i\in I\} $ of $M$ such that $||e_i||=1$ for all $i \in I$. Show that for any sequence $(a_i)_{i\in I}$ of elements of $A$ with $\lim_{i\rightarrow \infty} a_i=0$, the sum $\sum_i a_i e_i$ converges in $M$. 
\end{exer}
\begin{defn} Let $A$ be a Banach--Tate ring. A Banach $A$-module $M$ is called \emph{orthonormalizable}, or \emph{ONable} for short, if there exists a subset $\{e_i: i\in I \}$ of $M$ such that 
\begin{itemize}
\item Every element $m\in M$ can be written uniquely as a sum $\sum_{i\in I} a_i e_i$, with $a_i\in A$ and $\lim_{i\rightarrow \infty} a_i=0$. 
\item If $m= \sum_{i\in I} a_i e_i$, then $||m||= \max_{i\in I} |a_i|$. 
\end{itemize}
Such a subset $\{e_i:  i \in I\}$ is called an \emph{orthonormal basis} or \emph{ON basis} of $M$.\footnote{In $p$-adic functional analysis, the concept of a Hilbert space does not make sense, but ONable Banach modules are a good replacement for them.}
 
\end{defn}
Note that the second condition in the above definition implies that $||e_i||=1$ for all $i\in I$. 
Let $c_A(I)$ be the $A$-module of functions $f\colon I \rightarrow A$ such that $\lim_{i\in I} f(i)=0$, with addition and $A$-action defined point-wise. Then $||f|| = \max_{i\in I} |f(i)|$ defines a norm on $c_A(I)$ and turns it into a Banach $A$-module. It is ONable with canonical ON basis given by the functions $e_i\colon I\rightarrow A$, where $ e_i(j)= 0$ for $i\neq j$ and $e_i(i)=1$. 

\begin{exer} Show that a Banach $A$-module $M$ is ONable if and only if there exists a set $I$ and an $A$-linear isometric (i.e., norm-preserving) isomorphism $M\cong c_A(I)$. Show that the choice of an ON basis is equivalent to the choice of such an isometry.
\end{exer}

\begin{ex} The classical Tate algebra $K\langle T \rangle$ over a non-archimedean field $K$ is orthonormalizable. (Note $K\langle T \rangle$ is actually just another name for $c_K(\mathbb{N}_0)$). An orthonormal basis is given by $\{T^i: i \in \mathbb{N}_0\}$.
\end{ex}

We can use ON bases to associate a ``matrix'' to a homomorphism of Banach \linebreak$A$-modules.
For that let $\phi\colon M\rightarrow N$ be a homomorphism of ONable Banach \linebreak $A$-modules $M$ and $N$ with ON bases $\{e_i: i\in I\}$ and $\{f_j: j\in J\}$. We define the matrix coefficients $(a_{i,j})_{i\in I, j\in J}$ by 
\[\phi(e_i) = \sum_{j\in J} a_{i,j} f_j.\]
\begin{exer}\label{coeff}
\begin{enumerate}
\item Check that for all $i\in I$, we have $\lim_{j\rightarrow \infty} a_{i,j}=0$. 
\item Show that there exists a constant $C\in \mathbb{R}$ such that $|a_{i,j}| \leq C$ for all $i\in I,j\in J$. 
\item Show that in fact we have $||\phi||= \sup_{i,j} |a_{i,j}|$. 
\end{enumerate}
\end{exer}
Conversely, given a collection of $(a_{i,j})_{i\in I, j\in J}$ of elements of $A$ satisfying (1) and (2) of the exercise above, there is a unique continuous $\phi \colon M\rightarrow N$ with norm $||\phi||=\sup_{i,j} |a_{i,j}|$ and associated matrix $(a_{i,j})$. \\
In particular, we can measure the distance of two homomorphisms using their matrix coefficients as follows. Namely, let $\phi$ and $\psi$ be elements of $\Hom_{A}(M,N)$ with matrices $(a_{i,j})$ and $(b_{i,j})$. Then $||\phi-\psi|| \leq \epsilon$ if and only if $|a_{i,j}-b_{i,j}|\leq \epsilon$ for all $i,j$.

\subsection{Compact operators on Banach modules}

\begin{defn} Let $A$ be a Banach--Tate ring and let $M$ and $N$ be Banach \linebreak $A$-modules. A homomorphism $M\rightarrow N$ is said to be of \emph{finite rank} if its image is contained in a finitely generated $A$-submodule of $N$. \\
A homomorphism $\phi\colon M\rightarrow N$ is called \emph{compact} if it is a limit of finite rank homomorphisms in (the Banach $A$-module) $\Hom_A(M,N)$. 
\end{defn}
%\todo{mention Serre prop. 5 ? Char. of compact operators?}
\begin{rem}\begin{itemize}
\item In the literature compact operators also go by the name \emph{completely continuous maps}.
\item When $A=K$ is a field, compact operators have been studied in greater generality, i.e., not just for $K$-Banach spaces, but for more general topological $K$-vector spaces. A reference for results in this direction is the book \cite{schneider}.
\end{itemize} 
\end{rem}
 \medskip 
\noindent \textbf{Notation:} Assume from now on that $A$ is a \emph{noetherian} Banach--Tate ring. 
 \medskip 

The noetherian assumption guarantees the following. 
\begin{lem} \label{om} Let $A$ be a noetherian Banach--Tate ring. 
Every finitely generated \linebreak $A$-module $M$ has (up to equivalence) a unique norm making it into a Banach $A$-module. Any $A$-linear map $f\colon M\rightarrow N$ of finite Banach $A$-modules is continuous, $f(M) \subseteq N$ is closed and the induced map $M \rightarrow f(M)$ is open. \cite[Proposition II.4.2.2.]{mor}
\end{lem}

Our aim now is to study compact operators on ONable Banach $A$-modules over a noetherian Banach--Tate ring $A$. In particular, we want to show that a compact endomorphism has a \emph{characteristic power series}. 
\begin{rem} \begin{itemize}
\item A natural proof strategy for statements about compact operators is to first prove the statement for finite rank operators and then pass to the limit. In these situations one often has to choose a finitely generated submodule $P\subset M$ of some Banach-module $M$ and one would like to have that such a $P$ is automatically complete. The noetherian assumption on the Banach--Tate ring guarantees this (cf.\ Lemma \ref{2.3} below).
\item  In \cite[Section 3.1]{bellaiche}, Bellaïche relaxes the noetherian assumption. He considers Banach $K$-algebras $A$ with the property that every finitely generated submodule of an ONable Banach $A$-module is closed (\cite[Hypothesis 3.1.8]{bellaiche}). This is enough to prove part (3) of Lemma \ref{2.3} below (cf.\ \cite[Lemma 3.1.12]{bellaiche}), which is all that is needed to construct the characteristic power series. We can pass to this more general setting here as well, (i.e., assume that our Banach--Tate ring $A$ satisfies Hypothesis 3.1.8.) and all results of this section still hold. We have decided to work with noetherian $A$ here, as we will anyways do so from Section \ref{sec:specvar} onwards. 
\end{itemize} 
\end{rem}

Let $M$ be an ONable Banach $A$-module with ON basis $\{e_i: i \in I\}$. If $S\subseteq I$ is a finite subset, then we define $A^S$ to be the submodule $\oplus_{i\in S} Ae_i$ of $M$. 
Consider the projection 
\[\pi_S\colon M\rightarrow A^S, \ \sum_{i\in I} a_ie_i \mapsto \sum_{i\in S} a_ie_i.\]
Note that this is norm-decreasing onto a closed subspace of $M$. We have the following lemma.
\begin{lem}\cite[Lemma 2.3]{buz} \label{2.3} Let $M$ be an ONable Banach $A$-module with basis \linebreak $\{e_i: i\in I\}$ and let $P$ be a finite submodule of $M$. 
\begin{enumerate}
\item There is a finite subset $S\subseteq I $ such that $\pi_S\colon M \rightarrow A^S$ is injective when restricted to $P$.
\item $P$ is a closed subset of $M$ and hence is complete.
\item For all $\epsilon >0$ there is a finite set $T\subseteq I$ such that for all $p\in P$ we have \linebreak $||\pi_T(p)-p|| \leq \epsilon ||p||$. 
\end{enumerate}
\end{lem}
\begin{exer} Prove Lemma \ref{2.3} by checking that the proof in \cite{buz} goes through using Lemma \ref{om} above in the relevant places and substituting $\varpi$ for $\rho$. 
\end{exer}

\begin{prop} \label{2.4} Let $A$ be as above, i.e., a noetherian Banach--Tate ring. Let $M$ and $N$ be ONable Banach $A$-modules with ON bases $\{e_i: i\in I\}$ and $\{f_j: j\in J\}$. Let $\phi \colon M\rightarrow N$ be a continuous $A$-module homomorphism with matrix $(a_{i,j})$. Then $\phi$ is compact if and only if $\lim_{j \rightarrow \infty} \sup_{i\in I} |a_{i,j}|=0$. 
\end{prop}
\bigskip 

\begin{proof} (\cite[Prop.2.4]{buz} verbatim.) If the matrix of $\phi$ satisfies $\lim_{j \rightarrow \infty} \sup_{i\in I} |a_{i,j}|=0$, then for any $\epsilon >0$ there is a finite subset $S\subseteq J$ such that $||\phi-\pi_S \circ \phi|| \leq \epsilon$. This implies that $\phi$ is compact.\\
For the converse, first note that it suffices to show the result for operators of finite rank. If $\phi=0$ then the result is trivial, so assume $\phi \neq 0$ and $\phi(M)\subseteq P$ where $P\subseteq N$ is finitely generated. Then by part (3) of the previous lemma, for any $\epsilon >0$  we may choose a finite subset $T \subseteq J$ such that for all $p \in P$, $||\pi_T (p)-p|| \leq \epsilon ||p||/||\phi||$ and hence $||\pi_T\circ \phi -\phi|| \leq \epsilon$. Therefore $|a_{i,j}| \leq \epsilon $ if $j\notin T$ and as $\epsilon$ was arbitrary this implies the claim. 
\end{proof}

The following exercise is very important in the context of eigenvarieties. Here is why: When applying the eigenvariety machine to a given situation, one requirement is that one has a compact operator at hand. In practice, this is a Hecke operator (often called $U_p$) acting on some space of $p$-adic automorphic forms. In order to apply the machine one first has to check that $U_p$ indeed defines a compact operator. In most instances this is done by relating $U_p$ to some version of the following restriction map. 
\begin{exer} Consider the Banach $\Q_p$-algebra (the global functions on the closed disk of radius $p$)
\[\Q_p\langle pT\rangle = \left\{\sum_n a_n (pT)^n : \lim_{n\rightarrow \infty} a_n =0\right\}\subset \Q_p\llbracket T\rrbracket, \]
 equipped with the norm $||\sum_n a_n (pT)^n||= \max_n |a_n|$. 
Show that the ``restriction map''\footnote{The name comes from the geometric context: The unit disk $ \Spa(\Q_p\langle T \rangle, \Z_p\langle T \rangle)$ is contained in the disk $\Spa(\Q_p\langle pT\rangle, \Z_p\langle pT \rangle)$ of radius $p$ and this natural inclusion corresponds to the map $\mathrm{res}$, which we can think of as taking a function on the larger disk and restricting it to the smaller disk.} 
$$\mathrm{res}\colon \Q_p\langle pT\rangle \hookrightarrow \Q_p \langle T \rangle, \ \sum_n a_n (pT)^n \mapsto \sum_n a_n(pT)^n=\sum_n (a_n p^n) T^n$$ is a compact operator. (Here, as above, $\Q_p\langle T\rangle$ is equipped with the Gauß norm.)
\end{exer}

Now let $M$ be a Banach $A$-module with ON basis $\{e_i: i\in I\}$. Let $\phi\colon M\rightarrow M$ be a compact endomorphism with matrix $(a_{i,j})_{i,j\in I}$. If $S$ is a finite subset of $I$, let $c_S:= \sum_{\sigma \colon S\rightarrow S} \mathrm{sgn}(\sigma)\prod_{i\in S} a_{i,\sigma(i)}$, where the sum ranges over all permutations of $S$. For $n\geq 0$ define 
\[c_n:= (-1)^n\sum_S c_S,\]
where the sum is over all subsets $S\subseteq I$ of size $n$.\footnote{Note that $c_0= c_\emptyset =1$.} This converges in $A$ by \linebreak Proposition \ref{2.4}. 
\begin{defn} Let $\phi \colon M\rightarrow M$ be a compact endomorphism as above. 
We define the \emph{characteristic power series} (also called \emph{Fredholm determinant}) of $\phi$ as 
\[\det(1-X\phi):= \sum_{n\geq 0} c_n X^n \ \in A\llbracket X\rrbracket.\]
\end{defn}
\begin{rem}
Note that when $M$ is a finite-dimensional vector space over a field $A=K$, then $\det(1-X\phi)$ agrees with the algebraically defined determinant. % \todo{Remark about finite rank case/linear algebra}
\end{rem}
%\[\det(1-XU|M) = \lim_{J\subset I \text{finite}} \det (1-X(a_i,j)_{(i,j)\in J\times J}). \]
\begin{prop} Let $\phi\colon M \rightarrow M$ be a compact operator on an ONable Banach \linebreak $A$-module $M$ with ON basis $\{e_i: i\in I\}$. The characteristic power series $$\det(1-X\phi):= \sum_n c_n X^n \ \in A\llbracket X\rrbracket $$ of $\phi$ defined above converges for all $a\in A$. 
\end{prop}
\begin{proof} Let $r_1,r_2,\dots$ be the sequence of real numbers obtained from ordering the set $\{r_j(\phi):= \sup_{i\in I}|a_{i,j}|: j\in I\}$ decreasingly, so by Proposition \ref{2.4}, we have \linebreak $\lim_{k\rightarrow \infty} r_k= 0$. If $S\subseteq I$ is a subset of size $n$, then each product $a_{S,\sigma}:= \prod_{i\in S} a_{i,\sigma(i)}$ satisfies $|a_{S,\sigma}|\leq r_1 \cdots r_n$ and hence $|c_n|\leq r_1 \cdots  r_n$. If $\alpha \in \mathbb{R}_{> 0}$ is any positive real number, then 
\[|c_n|\alpha^n \leq (r_1\alpha ) \cdots (r_n\alpha). \]
As the $r_i\alpha$ tend to zero, so does $|c_n|\alpha^n$, which shows the claim. 
\end{proof}

In particular we see that the characteristic power series belongs to the ring of entire power series which is defined as follows.
\begin{defn} The ring $A\{\{X\}\}$ of \emph{entire power series} is defined as
\[ A\{\{X\}\}:= \left\{\sum_n a_n X^n \in A\llbracket X\rrbracket : \ \ \forall r\in \mathbb{Z},  \lim_{n\rightarrow \infty} a_n \varpi^{rn} = 0 \right\}.
\]
\end{defn}
\begin{exer}
Check that this definition does not depend on the choice of $\varpi$. 
\end{exer}

From the definition of the characteristic power series, it is a priori not clear that it is independent of the choice of ON basis. Verifying this is slightly tedious but important. In a similar spirit, we would like to know that the definition of the characteristic power series does not depend on the choice of an equivalent norm on $M$ (and even on $A$). The rest of this subsection is concerned with checking these technical properties. For this we need a lemma. 

\begin{lem}\label{2.5} Let $M$ be an ONable Banach $A$-module with ON basis $\{e_i : i\in I\}$.
\begin{enumerate}
\item If $\phi_n\colon M \rightarrow M, n= 1, 2, \dots$ is a sequence of compact operators that converges to a compact operator $\phi$, then $\lim_n \det(1-X\phi_n)= \det (1-X\phi)$, uniformly in the coefficients (in the norm of $A$).
\item If $\phi\colon M\rightarrow M $ is compact, and if the image of $\phi$ is contained in $P:= \bigoplus_{i\in S} A e_i$ for $S$ a finite subset of $I$, then $\det( 1-X\phi)= \det(1-X\phi|_P)$, where the right side is the usual algebraically-defined determinant.
\item (Strengthening of (2).) If $\phi \colon M\rightarrow M $ is compact and if the image of $\phi$ is contained in an arbitrary submodule $Q$ of $M$ which is free of finite rank, then $\det(1-X\phi)= \det(1-X\phi|_Q)$, where again the right side is the usual algebraically-defined determinant. 
\end{enumerate}
\end{lem}

\begin{proof} Proof of (1) (cf. \cite[Prop.8]{serre}): Let $(a_{i,j})$ and $(a_{i,j}^{(n)})$ be the matrices of $\phi$ and $\phi_n$ and put 
\[\det(1-X\phi)= \sum_m c_mX^m, \ \ \det(1-X\phi_n)= \sum_m c_m^{(n)}X^m.\]
Let $\varepsilon \in \R$ be such that $0<\varepsilon <1$. As before let $\{r_j(\phi):= \sup_{i\in I}|a_{i,j}|: j\in I\}$ for $j\in I$ and let $r_1,\dots, r_h$ be the subset of  $\{r_j(\phi):j\in I\}$ of elements $>1$. Fix a real number $\eta \in \mathbb{R}_{>0}$ such that 
\[\eta r_1\cdots r_h \leq \varepsilon.\]
Let $n$ be such that $||\phi-\phi_n|| \leq \eta$. Then $|a_{i,j}-a^{(n)}_{i,j}| \leq \eta$ for all $i,j$ (see Exercise \ref{coeff} and the paragraph after). Hence $r_j(\phi_n)=r_j(\phi)$ if $r_j(\phi)>\eta$ and $r_j(\phi_n)\leq \eta$ otherwise. For a finite set $S\subseteq I$, let us now consider
\[a_{S,\sigma}-a_{S,\sigma}^{(n)}= \prod_{i\in S} a_{i,\sigma(i)} -\prod_{i\in S} a_{i,\sigma(i)}^{(n)}.  \]
Rewriting this expression, one can deduce that 
\[|a_{S,\sigma}-a_{S,\sigma}^{(n)}| \leq \eta \sup_{i\in S} \prod_{j\in S, j\neq i} \sup(r_j(\phi),r_j(\phi_n)).\]
With the inequalities above and the fact that $\eta <1$ one sees that 
$|a_{S,\sigma}-a_{S,\sigma}^{(n)}|\leq \eta r_1\cdots r_h \leq \varepsilon$. Taking the sum over all $S$ we conclude that 
\[|c_m-c_m^{(n)}|\leq \varepsilon, \]
for all m.\\
For the proof of (2) and (3), the reader can copy the proof of \cite[Lemma 2.5 (b),(c)]{buz} verbatim, replacing Lemma 2.3 in loc.cit by Lemma \ref{2.3} above.
%Proof of (2) If $(a_{i,j})$ is the matrix of $\phi$, then $a_{i,j}=0$ for $j\notin S$ and the result follows directly from the definition of the characteristic power series. \\
%Proof of (3) Fix $\varepsilon > 0$. By Lemma \ref{2.3}(3), there is a finite set $T\subseteq I$ such that $\pi_T\colon Q\rightarrow P:= A^T$ has the property that $|\pi_T-i|\leq \varepsilon$, where $i\colon Q\rightarrow M$ is the inclusion. Define $\phi^T=\pi_T\phi\colon M\rightarrow P\subseteq M$. By (2) we see that $\det(1-X\phi^T)$ equals the algebraically-defined polynomial $\det(1-X\phi^T|_P)$. This polynomial is the same as the algebraically-defined $\det(1-X\phi_T)$, where $\phi_T=\phi \pi_T$ for $\phi\colon P \rightarrow Q$ and $\pi_T\colon Q\rightarrow P$. One can compute this latter determinant with respect to an arbitrary algebraic $A$-basis of $Q$. By Lemma \ref{2.3}(2), $Q$ with its subspace topology is complete and hence the topology on $Q$ is the Banach topology. Now as $\varepsilon$ tends to zero, $\phi_T\colon Q\rightarrow Q$ tends to $\phi\colon Q\rightarrow Q$ and $\phi^T\colon M\rightarrow M $ tends to $\phi\colon M\rightarrow M$, and the result follows from part (1). 
\end{proof}

\begin{prop} Let $(M,||\cdot||)$ and  $(M,||\cdot||')$ be Banach $(A,|\cdot|)$-modules. 
Suppose that the norms $||\cdot||$ and $ ||\cdot||'$ are equivalent on $M$, i.e., they induce the same topology and that they both turn $M$ into an ONable Banach $A$-module.
Let $\phi\colon M\rightarrow M $ be an $A$-linear endomorphism. Then $\phi $ is compact with respect to $||\cdot||$ if and only if it is compact with respect to $||\cdot||'$ and furthermore if $\{e_i: i \in I\}$ and $\{ f_j : j \in J\}$ are ON bases for $(M, ||\cdot||)$ and $(M, ||\cdot||')$ respectively, then the characteristic power series $\det(1-X\phi)$ with respect to these bases coincide. 
\end{prop}
\begin{proof} Note that by Remark \ref{eqnorms}, the induced operator norms on $\Hom_A(M,M)$ are equivalent. The first claim then follows as the set of compact operators only depends on the topology of $\Hom_A(M,M)$. We further claim that $\phi$ can be written as the limit of operators $\phi_n$, which have image contained in \emph{free} modules of finite rank. 
Let $\varepsilon > 0$ and consider the matrix $(a_{i,j})$ of $\phi$ in the ON basis $\{e_i: i \in I\}$. Then just as in the proof of Proposition \ref{2.4} we may find a finite set $T\subseteq I$ such that $||\pi_T\circ\phi - \phi || \leq \varepsilon. $ Hence, for a suitable choice of sequence $(T_n)_n$ of finite subsets $T_n$ of $I$ we see that $\phi_n:=\pi_{T_n}\circ \phi$ converges to $\phi$ in either norm, i.e., $||\phi_n - \phi || \rightarrow 0$ and $\||\phi_n - \phi ||'\rightarrow 0$ as $n\to\infty$. By definition, the image of any $\phi_n$ is contained in a finite free submodule of $M$. Hence we can apply the previous lemma (Parts (1) and (3)) to $\phi= \lim \phi_n$ to deduce equality of the characteristic power series. 
\end{proof}

%\begin{exer}
%Let $I$ be a set. Let $A$ be as above and let $|\cdot|$ and $|\cdot|'$  be two norms on $A$ that are equivalent. Then $(c_A(I), |\cdot|) \cong (c_A(I), |\cdot|')$ as topological $A$-modules via the identity map.  
%\end{exer}
We want as much flexibility as possible in changing norms, so let us show something even more general: 
\begin{prop}\label{ind} Let $(M,||\cdot||)$ be a Banach $(A,|\cdot|)$-module and let $(M,||\cdot||')$ be a Banach $(A,|\cdot|')$-module. Suppose the norms $|\cdot|$ and $|\cdot|'$ are equivalent on $A$, that the norms $||\cdot||$ and $ ||\cdot||'$ induce the same topology on $M$ and $M$ is ONable with respect to both norms. 
%Then $M$ is ONable with respect to the second norm as well. 
Let $\phi\colon M\rightarrow M $ be an $A$-linear endomorphism. Then $\phi $ is compact with respect to $||\cdot||$ if and only if it is compact with respect to $||\cdot||'$ and furthermore if $\{e_i: i \in I\}$ and $\{ f_j : j \in J\}$ are ON bases for $(M, ||\cdot||)$ and $(M, ||\cdot||')$ respectively, then the characteristic power series $\det(1-X\phi)$ with respect to these bases coincide. 
\end{prop}
\begin{proof} This follows just like the last proposition, once we remark that given equivalent norms on $A$ and $M$ as in this proposition, the two resulting operator norms on $\Hom_A(M,M)$ define the same topology. This topology can be described independent of any norm, as the so called \emph{strong topology} or the \emph{topology of bounded convergence}. We refer to \cite[Chapter 1.6]{schneider}, in particular to Remark 6.7, where it is shown that the strong topology agrees with the topology of any operator norm. 
\end{proof}

% \begin{corollary} [2.6] Let $M$ be an $A$-module, and let $|\cdot|_1$ and $|\cdot|_2$ be norms on $M$ both turning $M$ into an ONable Banach $A$-module, and both inducing the same topology on $M$. Then an $A$-linear map $\phi\colon M\rightarrow M$ is compact with respect to $|\cdot|_1$ if and only if it is compact with respect to $|\cdot |_2$ and furthermore if $\{e_i: i \in I\}$ and $\{ f_j : j \in J\}$ are ON bases for $(M, |\cdot|_1)$ and $(M, |\cdot|_2)$ respectively, then the definitions of $\det(1-X\phi)$. with respect to these bases coincide. 
% \end{corollary}
% \begin{prop} The power series $\det(1-XU)$ is independent of the choice of basis. Corollary 2.6 in Buzzard. 
%\end{prop}
We record the following lemma, which is useful in practice when changing modules. 
 \begin{lem}\label{2.7} If $M$ and $N$ are ONable Banach $A$-modules and if $\phi \colon M\rightarrow N$ is compact and $v\colon N \rightarrow M$ is continuous, then $\phi \circ v$ and $v \circ \phi$ are compact and $\det(1-X (\phi \circ v))= \det (1-X (v \circ \phi))$. 
 \end{lem}

 \begin{proof} We leave it as an exercise to check that the proof of Lemma 2.7. in \cite{buz} goes through.
\end{proof}

\pagebreak

\subsection{More flexibility in the setup}
In practice, the above setup is too rigid. In this section we discuss two degrees of flexibility. Firstly, we need to work with a more general class of $A$-modules. We can generalize the construction of characteristic power series from ONable Banach $A$-modules to Banach $A$-modules with property (Pr) (see Definition \ref{(pr)} below). This generalization turns out to be a  straightforward application of Proposition \ref{ind}. Secondly, we want to apply the result of this section in a \emph{geometric situation} involving the affinoid adic space $\Spa(A,A^+)$.  %where our $A$-module $M$ will be beefed up to a sheaf on an affinoid adic space $\Spa(A,A^+)$. 
In such geometric situations it is important that one can control base change, or at the very least understands how to pass from the situation over $\Spa(A,A^+)$ to a rational subdomain. Hence we discuss some results on changing the underlying Banach--Tate ring. 

\subsubsection*{Compact operators on Banach modules with property (Pr)}
%We keep the notation from the previous section, so $A$ is a noetherian Banach--Tate ring. 
For the next definition recall the Banach $A$-module $c_A(I)$ attached to a set $I$ from Section \ref{sec:onbs}. 
\begin{defn}\label{(pr)}
\begin{enumerate}
\item  A Banach $A$-module $M$ is called \emph{potentially ONable} if there exists a set $I$ such that $M$ is $A$-linearly homeomorphic to $c_A(I)$. 
%there exists a norm on $M$ equivalent to the given norm, for which $M$ becomes an ONable Banach $A$-module. 
A set in $M$ corresponding to $\{e_i=(\delta_{i,j})_j: i\in I\}$ under such a map is called a \emph{potential ON basis}. 
\item We say a Banach $A$-module has \emph{property (Pr)} if it is a direct summand of a potentially ONable Banach $A$-module.
\end{enumerate}
\end{defn}
\begin{rem}
Note that the notions of being potentially ONable and of having property (Pr) are independent of the chosen norms. They only depend on the underlying topological $A$-module structure. 
\end{rem}

\begin{ex} 
\begin{enumerate}
\item Let $A=\Q_p$ and $M= \Q_p(\sqrt{p})$ equipped with its usual norm, then $||M||\neq |A|$ and so $M$ is not ONable, but it is potentially ONable. 
\item Let $K$ be a discretely valued non-archimedean field. Then in fact \emph{any} Banach space over $K$ is potentially ONable (cf.\ Proposition 1 of \cite{serre} and the remarks before it, as well as \cite[Section 3.1.4]{bellaiche}). 
\end{enumerate}
\end{ex}
A Banach $A$-module $P$ has property (Pr) if and only if for every surjection $f\colon M\rightarrow N$ of Banach $A$-modules and for every continuous map $\alpha\colon P\rightarrow N$, $\alpha$ lifts to a map $\beta\colon P \rightarrow M$ such that $f\circ \beta = \alpha$. One can show this using the Open Mapping Theorem. Note that if $P=c_A(I)$ for some set $I$, then to give $\alpha \colon P\rightarrow N$ is to give a bounded map $I \rightarrow N$ and such a map lifts to a bounded map $I\rightarrow M$ by the open mapping theorem.
\pagebreak

\begin{exer} \label{proj}
\begin{enumerate}
\item Work out the details of the proof of the universal property.
\item Use the universal property to show that if $P$ is a finite Banach $A$-module with property (Pr), then $P$ is projective as an $A$-module.
\item Show that the converse is also true: If $P$ is a finite $A$-module which is projective as an $A$-module, then $P$ has property (Pr). 
\end{enumerate}
\end{exer}

If $M$ is potentially ONable then one still has the notion of characteristic power series of a compact operator on $M$. This can be defined by choosing a homeomorphism $M\cong c_A(I)$ and then transferring the norm of $c_A(I)$ to $M$ via this homeomorphism, which results in the choice of an equivalent ONable norm on $M$. Using this norm one then gets a characteristic power series as before and by Proposition \ref{ind} this is independent of all choices. \\
Even more generally: Say that $P$ satisfies property (Pr) and that $\phi\colon P\rightarrow P$ is a compact endomorphism. We can define $\det(1-X\phi)$ as follows: Choose $Q$ such that $P\oplus Q$ is potentially ONable and define $\det(1-X\phi)= \det (1-X(\phi\oplus 0))$. Note that $\phi \oplus 0 \colon P\oplus Q \rightarrow P \oplus Q$ is compact, as can be checked easily. This definition might a priori depend on the choice of $Q$, but let us check that it does not. If $R$ is another Banach $A$-module such that $P\oplus R$ is potentially ONable, then so is $P\oplus Q \oplus P \oplus R$ and then the maps $\phi \oplus 0 \oplus 0 \oplus 0 $ and $0 \oplus 0 \oplus \phi \oplus 0 $ are conjugate via an isometric $A$-module isomorphism, and therefore have the same characteristic power series. We can conclude that $\det(1-X\phi)$ is well-defined, i.e., independent of the choice of $Q$ by doing and applying the following exercise. 

\begin{exer} Let $M$ and $N$ be ONable $A$-modules and $\phi\colon M\rightarrow M$ a compact endomorphism, then the characteristic power series of $\phi$ and of $\phi\oplus 0 \colon M\oplus N\rightarrow M\oplus N$ coincide. 
\end{exer}
%We mention without proof that Lemma \ref{2.7} holds more generally for Banach $A$-modules with property (Pr).\todo{exercise?} 

\subsubsection*{Controlling base change} 
We briefly give some recollection on tensor products. As before, let $A$ be a noetherian Banach--Tate ring. 
Let $M$ and $B$ be two normed $A$-modules. Define a function 
\[|\cdot| \colon M\otimes_A B\rightarrow \mathbb{R}_{\geq 0}, |g|:= \inf \left(\max_{1\leq i\leq r } |m_i||b_i|\right), \]
where the infimum ranges over all possible representations $g=\sum_{1=i}^r m_i\otimes b_i$, with $m_i\in M$ and $b_i \in B$ of $g$.
This turns $M\otimes_A B$ into a semi-normed $A$-module, and as such it has a completion (see \cite[1.1.7]{BGR}), which is a normed $A$-module and which is denoted by $M\widehat{\otimes}_A B$.
If $h \colon A\rightarrow B$ is a contractive (i.e., $ |h(a)|\leq |a|$ for all $a\in A$) morphism of Banach--Tate rings, then $B$ is a normed $A$-module and then $M\widehat{\otimes}_A B$ is a normed $B$-module.

Given two complete Tate rings $A$ and $B$ and a continuous map $h\colon A \rightarrow B$, the image of a pseudo-uniformizer is a pseudo-uniformizer. Choose a pseudo-uniformizer $\varpi\in A$ and compatible rings of definitions. Equip $A$ and $B$ with the standard norms as in (\ref{standnorm}). Then $h$ is contractive.\footnote{In particular, with these choices $h$ is bounded. Note that for a general map between Banach--Tate rings \emph{continuity} is a weaker condition than \emph{boundedness}, as we have already mentioned in the context of changing the norm on $A$ and the identity map.} 

\begin{exer} Let $A$, $B$ and $C$ be complete Tate rings, with continuous maps $A\rightarrow B$ and $A \rightarrow C$. Compare the above completed tensor product to the construction from Bergdall's lectures (\cite[Section 3.1.2]{bergdall}).
\end{exer}
Let us verify that the concepts introduced above behave well with respect to completed tensor products. 

\begin{lem} \label{2.8} Let $h\colon A\rightarrow B$ be a continuous morphism of noetherian Banach--Tate algebras and equip $A$ and $B$ with the standard norms (so that by the above, $h$ is contractive). Then the following holds.
\begin{enumerate}
\item If $M$ is a potentially ONable Banach $A$-module, then $M\widehat{\otimes}_AB$ is a potentially ONable Banach $B$-module. Furthermore, if $\{e_i: i\in I\}$ is a potential ON basis for $M$ then $\{e_i\otimes 1: i\in I\}$ is a potential ON basis for $M\widehat{\otimes}_AB$.
\item If $M$ has property $(Pr)$ then so does  $M\widehat{\otimes}_AB$.
 \end{enumerate} 
\end{lem}
\begin{proof} We prove part (1). Set $N=c_B(I)$ and let $\{f_i: i\in I\}$ be its canonical ON basis. Then there is a natural $A$-bilinear bounded map $M\times B\rightarrow N$ that sends $(\sum_i a_ie_i,b)$ to $\sum_i bh(a_i)f_i$ which induces a continuous map $\varphi\colon M\widehat{\otimes}_A B\rightarrow N$.\footnote{This follows from the universal property of the completed tensor product, see \cite[2.1.7 Proposition 1]{BGR} for details.} On the other hand if $n\in N$, one can write $n$ as a limit of elements of the form $\sum_{i\in S} b_i f_i$, where $S$ is a finite subset of $I$. The element $\sum_{i\in S} e_i\otimes b_i$ of $M\otimes_A B$ satisfies
\[|\sum_{i\in S} e_i\otimes b_i|\leq C \max_{i\in S}|b_i|\]
and hence as $S$ gets bigger the resulting sequence is Cauchy and so its image in $M\widehat{\otimes}_A B$ tends to a limit. This gives a well-defined continuous $A$-module homomorphism $N\rightarrow M\widehat{\otimes}_AB $ which is inverse to $\phi$. \\
Part (2) is left as an exercise.
\end{proof}

\begin{corollary}\label{bccomp} Let $h\colon A\rightarrow B$ be a continuous morphism of noetherian Banach--Tate algebras. 
\begin{enumerate}
\item If $M$ and $N$ are potentially ONable Banach $A$-modules with potentially ON bases $\{e_i:i\in I\}$ and $\{f_j: j\in J\}$ and $\phi\colon M \rightarrow N$ is compact, with matrix~$(a_{i,j})$, then $\phi\otimes 1 \colon M\widehat{\otimes}_A B \rightarrow N\widehat{\otimes}_A B$ is also compact and if $(b_{i,j})$ is the matrix of $\phi\otimes 1$ with respect to the bases $(e_i\otimes 1) $ and $(f_j\otimes 1)$ then $b_{i,j}=h(a_{i,j})$ for all $i \in I$ and $j\in J$. 
\item If $M$ and $N$ have property (Pr) then again, $\phi \otimes 1$ is compact. 
\end{enumerate} 
\end{corollary}
\begin{proof} Compactness of $\phi \otimes 1 $ in both cases follows from Proposition \ref{2.4} and the rest is left as an exercise. 
\end{proof}
\noindent One immediately checks:
\begin{corollary} \label{bcchar} For $h\colon A\rightarrow B$ a continuous morphism of noetherian Banach--Tate algebras, $M$ a Banach A-module satisfying propery $(Pr)$ and $\phi$ a compact endomorphism of $M$ the following holds.  If $\det(1-X\phi)= \sum_n c_n X^n$, then the characteristic power series of $\phi \otimes 1 $ on $M\widehat{\otimes}_A B$ is given by $\det(1-X (\phi\otimes 1))= \sum_n h(c_n) X^n$.
\end{corollary}

\section{Spectral Varieties} \label{sec:specvar}
In this section we construct and study so called spectral varieties over an affinoid base. As mentioned in the introduction, in our eigenvariety construction, the spectral variety sits between the eigenvariety and the weight space. There is a structural map from the spectral variety to weight space whose geometric properties we also study in this section. These properties will be used in the construction of the eigenvariety in Section \ref{sec:evs} below. The main references for this section are \cite[Appendix B]{aip} and \cite[Section 2.3]{JN1}. 

\subsection{Construction of spectral varieties}
Let $(A,A^+)$ be a Tate--Huber pair %which is also a $\Z_p$-algebra 
with pseudo-uniformizer $\varpi \in A$. Assume throughout this section that $A$ is strongly noetherian, so in particular $(A,A^+)$ is sheafy.
Let $W:=\spa(A,A^+)$ be the associated affinoid adic space and consider the relative affine line $\mathbb{A}^1_W$ over $W$. 
Recall this is the increasing union of affinoid disks
\[ \mathbb{A}^1_W = \bigcup_{n\geq1} \mathcal{D}_n, \] 
where $\mathcal{D}_n= \Spa(A\langle \varpi^nX\rangle,A^+\langle \varpi^nX\rangle)$ is the affinoid disk of ``radius'' $\varpi^{-n}$. 
\begin{exer} Show that the global functions $\mc{O}_{\mathbb{A}^1_W }(\mathbb{A}^1_W )$ are given by the ring of entire power series $A\{\{X\}\}$. 
\end{exer}
For later purposes we define the affinoid open subsets $\mathcal{D}_{\lambda}\subseteq \mathbb{A}^1_W$ for any \linebreak $\lambda= l/m \in \mathbb{Q}$, with $l\in \Z$ and $m\in \mathbb{N}$ by 
\[\mathcal{D}_{\lambda}= \{x \in \mathbb{A}^1_W: \ |\varpi^lX^m(x)|\leq 1\}\subseteq \mathbb{A}^1_W.\]
Note that $ \mathbb{A}^1_W = \bigcup_{\lambda \in \Q} \mc{D}_{\lambda}$.
\begin{defn} An entire power series $F=\sum_{n=0}^\infty a_n X^n \in A\{\{X\}\}$ is called \emph{Fredholm series} if $a_0=1$. 
\end{defn}
For example, the characteristic power series of a compact operator as in the last section is Fredholm series. An entire power series $F\in A\{\{X\}\}$ gives rise to a coherent $\cO_{\mathbb{A}^1_W}$-ideal $\mathcal{I}_F$ as follows. As $F$ is naturally an element of $A\langle \varpi^nX\rangle$ for all $n\geq 1$, we can put $\mathcal{I}_F |_{\mathcal{D}_n}:=\widetilde{(F)}$, where $(F) \subset A\langle \varpi^nX\rangle$ is the ideal generated by $F$ and $\widetilde{(F)}$ is the associated coherent sheaf on $\mathcal{D}_n$ as in \cite[Section 1.2]{johansson}. This glues to a coherent ideal sheaf $\mathcal{I}_F$ on $\mathbb{A}^1_W$.

\begin{defn} Let $F\in A\{\{X\}\}$ be a Fredholm series. The \emph{spectral variety} (also called Fredholm hypersurface) of $F$ is defined as the closed adic subspace 
\[V(F):=V(\mathcal{I}_F) \subset \mathbb{A}^1_W.\]
\end{defn}
Let us describe this  more explicitly. On the affinoid disk $\mathcal{D}_n$  we have by construction that  
\[ V(F)\cap \mathcal{D}_n= \spa(B_n, B_n^+),\]
where $B_n:= A\langle \varpi^nX\rangle/(F)$ and $B_n^+:=(A\langle \varpi^nX\rangle/(F))^+$ is the integral closure of the ring $A^+\langle \varpi^nX\rangle$ in $B_n$. %We have a natural structure map $V(F)\rightarrow \Spa(A,A^+)$, which we shall study in a moment. 
 
The formation of the spectral variety behaves well with respect to base change in the following sense. Let \[h\colon (A,A^+) \rightarrow (B,B^+)\] be a continuous morphism of complete Tate--Huber pairs, with $A$ and $B$ strongly noetherian. We have corresponding natural maps 
$$Y:=\spa(B,B^+)\rightarrow W:=\Spa(A,A^+)$$ and $\bA^1_Y\rightarrow \bA^1_W$. The map $h\colon A\rightarrow B$ of the Tate rings induces a natural map $A\{\{X\}\} \rightarrow B\{\{X\}\}$. For a Fredholm series $F=\sum_n{a_n X^n} \in A\{\{X\}\} $ let 
$$h(F):= \sum_n h(a_n)X^n \in B\{\{X\}\}$$ be its image. 
Then the construction of the spectral variety satisfies the following property. 
\begin{lem}\label{specbc} Let $h\colon (A,A^+) \rightarrow (B,B^+)$ be as above. Let  $F=\sum_n{a_n X^n} \in A\{\{X\}\} $ be a Fredholm series and let $V(F)$ be its spectral variety. Then 
\[V(F)\times_W Y\cong V(h(F))\subset \bA^1_Y.\]
\end{lem}
\begin{proof} Exercise.
\end{proof}

We have a natural structure map $w\colon V(F)\rightarrow \Spa(A,A^+)$ whose properties we now study. We will show that $w$ is flat, locally quasi-finite and partially proper. (We refer to \cite[Lecture 2]{johansson} for definitions and background regarding these concepts.) 
These properties are very useful for the eigenvariety machine. We start by verifying flatness of $w$. 
Recall that a morphism $f\colon X\rightarrow Y$ of adic spaces is called flat, if all maps on stalks are flat, i.e., if for all $x\in X$, the natural morphism $\cO_{Y,f(x)} \rightarrow \cO_{X,x}$ is flat. If $X$ and $Y$ are affinoid and the map between the underlying Huber rings is flat, then $f$ is flat. 

\begin{lem} \label{4.1}  For any $n\geq 1$, the morphism $w|_{\mathcal{D}_n} \colon V(F) \cap \mathcal{D}_n \rightarrow \spa(A,A^+)$ is flat. Hence, the structure morphism $w$ is flat. 
\end{lem}
\begin{proof}
Assume first that $n=1$, so that we want to show that $A\langle X \rangle/(F)$ is flat. Note that as $A$ is a complete noetherian Tate ring, $A\langle X \rangle $ is flat over $A$ by \cite[8.30]{w}. 
As is explained in the same reference, it now suffices to show that for any finitely generated $A$-module $M$, the multiplication map 
\[w_F\colon M\langle X\rangle \rightarrow M\langle X \rangle, \ P\mapsto F\cdot P\] is injective. As the constant term of $F$ is equal to $1$, one can just check this directly. Hence $A\langle X\rangle/(F)$ is flat over $A$. The general case follows by rescaling. 
%Alternatively one can follow the proof of  \cite[Lemma 4.1]{buz}\todo{strong noetherian assumption?}\footnote{We remark that this uses the fact that $A\langle X \rangle$ is noetherian, the first proof above works in slightly bigger generality.}.
\end{proof}

Our next goal is to show that $w$ is locally quasi-finite. We follow the proof in \cite[Appendix B.1]{aip}. For that, we first observe that we can reduce to the case of $A=K$ being a non-archimedean field by applying the following lemma. 
\begin{lem}\label{1.5.2} Let $f\colon X \rightarrow Y$ be a morphism of analytic adic spaces which is of finite type.  Then $f$ is quasi-finite if and only if for every rank $1$ point $y$ of $Y$, the morphism $X\times_Y \Spa(k(y))\rightarrow \spa(k(y))$ is finite.
\end{lem}
\begin{proof} This is the equivalence of (a) and (f) in \cite[Lemma 1.5.2]{hu1}.
\end{proof}

In a nutshell, the field case will follow from the fact that an entire power series $F \in K\{\{X\}\}$, where $K$ is a non-archimedean field, has only finitely many zeroes on any disk $\mathcal{D}_n \subset \A^1_K$, which follows from Weierstrass theory. For the actual proof in the field case (Lemma \ref{L.B.1} below) we will make use of Newton polygons. For the convenience of the reader we have included a summary of the relevant background. 

\subsection{Interlude on Newton polygons}
Let $K$ be a non-archimedean field. We fix a pseudo-uniformizer $\varpi$ and the norm on $K$ to be $|a|=\inf \{ p^{-n} : a \in \varpi^n K^\circ, n \in \Z\}$. The corresponding (additive) valuation will be denoted by $v$. It satisfies $v(\varpi)=1$. 

Let $F=\sum a_n X^n \in K\llbracket X\rrbracket$ be a Fredholm series, so that $a_0=1\neq 0$.  
The Newton polygon of $F$ is constructed as follows. (For details on Newton polygons of power series we refer the reader to \cite[Chapter 7.4]{gouv}.) %We briefly define the Newton polygon of a power series $F\in K\llbracket T\rrbracket$, following \cite[\S 4.2]{as} (in this special case). 
In $\mathbb{R}_{\geq 0}\times (\mathbb{R}\cup \{+\infty\})$ plot the points $(n, v(a_n))$ for all $n$. If $a_n =0 $ then $v(a_n)=+\infty$ and we think of $(n,+\infty)$ as a point infinitely high up the $y$-axis. 
Let $\mathcal{C}(F)$ be the set of these points.  
Define the \emph{Newton polygon $\mathcal{N}(F)$} of $F$ to be the lower convex hull of $\mathcal{C}(F)$.\footnote{That is, put a pin at $(0, v(a_0))=(0,0)$, attach a long thread to the pin and then rotate this counterclockwise until you hit the next element of $\mathcal{C}(F)$. Then repeat this.} 

If $F$ is a polynomial, i.e. if there is an $m\geq 0$, such that $a_n=0$ for all $n\geq m$, then put $\mathcal{N}(F)(x)=+\infty$ for all $x>m_0$ where $m_0$ is the degree of $F$. %For $x < n_0$ define $\mathcal{N}(F)=+\infty$. \\

The Newton polygon $\mathcal{N}(F)\colon \mathbb{R}_{\geq 0} \rightarrow \mathbb{R}\cup \{+\infty\}$ is a piece-wise linear function. The linear pieces are called the \emph{segments} of $\mathcal{N}(F)$, the slopes of the segments are called the \emph{slopes} of $\mathcal{N}(F)$. 
Note that by construction the slopes of $\mathcal{N}(F)(x)$ increase as $x\in \mathbb{R}_{\geq 0}$ increases. 
Let $h\in \R$. We say that a Fredholm series $F\in K\{\{X\}\}$ has slope $\leq h$ (or $>h$) if the slopes of its Newton polygon $\mathcal{N}(F)$ are all $\leq h$ (or $>h$).

Recall that for a polynomial $P \in K[X]$, the slopes of its Newton polygon correspond to the valuations of its roots in $ \overline{K}$ (see \cite[Prop.II.6.3]{neukirch}). 
For an entire power series $F$ one combines Weierstrass theory with the theory of Newton polygons to get control over the zeros of $F$. Again, the Newton polygon encodes a lot of information, as is summarized in the following lemma. A reference for this lemma is \cite[Theorem 4.4.2]{as} (applied to the case where $A$ is a field). (A more elementary reference for the case when $K$ is a $p$-adic local field, see \cite[7.4.10 and 7.4.11]{gouv}.)

\begin{lem} Let $F=\sum_n a_n X^n \in K\{\{X\}\}$ be a Fredholm series. Let $s_1,\dots, s_k$ be the first $k$ slopes of the Newton polygon $\mathcal{N}(F)$. Let $n_0$ be the $x$-coordinate of the right endpoint of the $k$-th segment of the Newton polygon.  Then there exists a polynomial $P \in K[X]$ of degree $n_0$ and a power series $S \in K\llbracket X\rrbracket$ such that 
\begin{enumerate}
\item $F=PS$.
\item $S$ is entire.
\item The Newton polygon of $P$ is equal to the portion of the Newton polygon of $F$ contained in the region $0\leq x\leq n_0$.
\item $S$ has slope $>s_k$. 
\end{enumerate}
\end{lem}

% A subset $\mathcal{N}\subset \bR^{2}$ is said to be \emph{sup-convex} if it is convex and, whenever a point $(a,b)\in \mathcal{N}$, $\mathcal{N}$ contains the whole half-line $\{(a,b+t)\mid t\geq 0\}$ above it. Given an arbitrary subset $S\subset \bR^{2}$, there is a unique smallest sup-convex set containing $S$, which we will denote by $\mathcal{C}(S)$. If $I\subset \bZ_{\geq 0}$ and $\omega\, :\, I \rightarrow \bR$ is a function, then any set of the form $\mathcal{C}(\{(n,\omega(n))\mid n\in I\})$ is called a \emph{Newton polygon}.  I prefer lower convex hull!
 %We refer to \cite[\S 4.2]{as} for the notions of vertices, edges and slopes of a Newton polygon.
%
%\begin{defn}
%Let $F=\sum_{n\geq 0} a_{n}T^{n}\in K\llbracket T\rrbracket$. Fix a pseudo-uniformizer $\vp\in K$ and consider the corresponding valuation $v_{\vp}$. The \emph{Newton polygon of $F$} is the Newton polygon $\mc{C}(\mc{S}(F))$, where
%$$ \mc{S}(F)=\{(n,v_{\vp}(a_{n})) \mid n\in I_{F} \}\sub \R^{2}$$
%with $I_{F}=\{ n\in \Z_{\geq 0} \mid a_{n}\neq 0 \}$.
%\end{defn}
%
%Let $h\in \R$. We say that a power series $F\in K\llbracket T\rrbracket$ has slope $\leq h$ (or $>h$) if all slopes of its Newton polygon are $\leq h$ (or $>h$). Now consider a Banach--Tate ring $R$ with a multiplicative pseudo-uniformizer $\vp$. We say that $F\in R\llbracket T\rrbracket$ has slope $\leq h$ (or $>h$) if, for any $x$ in the Gelfand spectrum $\mc{M}(R)$ with residue field $K_{x}$, the specialization $F_{x}\in K_{x}\llbracket T\rrbracket$ has slope $\leq h$ (or $>h$).

\subsection{Properties of the structure map}

We can now show that $w$ is locally quasi-finite.

\begin{lem} \label{L.B.1} Suppose that $A=K$ is a non-archimedean field. Then there exists a strictly increasing sequence of rational numbers $(\lambda_n) \in \Q^{\mathbb{N}}_{>0}$ with $\lim \lambda_n = +\infty$ such that \linebreak $V(F)\cap \mathcal{D}_{\lambda_n}$ is finite over $\Spa(K,K^+)$. 
\end{lem}
\begin{proof} Note that if $F$ is a polynomial, the lemma is easy. So we may assume that $F$ is not a polynomial and so the Newton polygon $\mathcal{N}(F)$ has infinitely many segments and hence infinitely many slopes. Let $\{s_n\}_{n\in \mathbb{N}}$ be the slopes of $\mathcal{N}(F)$ of $F$, ordered strictly increasingly. For any $n\in \mathbb{N}$ pick $\lambda_n = l/m \in \Q$ such that $s_n < \lambda_n < s_{n+1}$. Consider the affinoid subspace $\mathcal{D}_{\lambda_n}=\{x : |\varpi^l X^m(x)|\leq 1\} \subset \bA^1_{\spa(K,K^+)}$ that we defined at the beginning of this section. We claim that $V(F)\cap \mathcal{D}_{\lambda_n} \rightarrow \spa(K,K^+)$ is finite.
After possibly passing to a finite field extension $K'$ we may assume that $K'$ contains an element $\beta$ with $v(\beta)=\lambda_n$.  By \cite[Proposition 1.4.9]{hu1} it suffices to show that the base change morphism
\[(V(F) \cap \mathcal{D}_{\lambda_n})\times_{\spa(K,{K}^+)} \spa(K',{K'}^+) \rightarrow \spa(K',{K'}^+)\]
is finite. So without loss of generality we may assume that $K$ contains an element~$\beta $ with $v(\beta)=\lambda_n$. 
Then $\mathcal{D}_{\lambda_n}= \Spa(K\langle \beta X\rangle, K^+\langle \beta X \rangle)$. Let $P$ be the polynomial whose Newton polygon is given by the section of $\mathcal{N}(F)$ of the first slopes $s_1,\dots,s_n$ and let $F(X)=P(X)\cdot S(X)$ be the corresponding factorization. Then $S$ is a Fredholm series with slopes strictly bigger than $\lambda_n$. In the algebraic closure $\overline{K}$ of $K$ we have a decomposition $P(X) = \prod (1-\alpha_iX)$ with $v(\alpha_i)< \lambda_n$ for all $i$ and we may rewrite 
\[P(X) = \prod \beta^{-1} \alpha_i \prod (\alpha_i^{-1}\beta - \beta X) \in \overline{K}[X].\] 
Moreover the polyomial in $\beta X$, $Q(X):= \prod (\alpha_i^{-1}\beta - \beta X)$ is an element of $ K^{\circ\circ}\langle \beta X\rangle$. Note also that (as $S(X)$ is a unit in $K\langle\beta X\rangle$)
\[K\langle\beta X\rangle /(F(X))=K\langle\beta X\rangle /(Q(X))=K\langle T\rangle /(Q'(T)), \]
where $T$ is the variable $\beta X$ and $Q'(T)= \prod (\alpha_i^{-1}\beta - T)$ and this is a finite $K$-algebra. Finally $(K\langle\beta X\rangle /(F(X)))^+= K^+\langle T \rangle /(Q'(T))$ and this is finite over $K^+$, which completes the proof. \\
\end{proof}

The following theorem is the main structural result on the spectral variety. 
%For that we assume from now on that $A$ is a \emph{strongly} noetherian Banach--Tate ring. 

\begin{thm} \label{T.B.1.} The morphism $w\colon V(F)\rightarrow W=\Spa(A,A^+)$ is flat, locally quasi-finite and partially proper. For all $x\in V(F)$ there is an open neighborhood $U$ of $x$ in $V(F)$ such that $\overline{\{x\}}\subset U$ and an open neighborhood $V$ of $w(x)$ in $W$ such that $\overline{\{w(x)\}}\subseteq V$, $w(U)\subset V$ and the induced morphism $w|_U\colon U\rightarrow V$ is finite and flat. 
\end{thm}
\begin{proof}
The morphism $w$ is locally of finite type and quasi-separated by construction. It is locally quasi-finite by Lemma \ref{1.5.2} %\cite[Lemma 1.5.2]{hu1} 
 and Lemma \ref{L.B.1}. It is partially proper again by Lemma \ref{L.B.1} and the valuative criterion of partial properness (\cite[Corollary 1.3.9]{hu1}). Now by Proposition \ref{1.5.6} below, $w$ being locally quasi-finite and partially proper implies that for every $x\in V(F)$ and $w(x) \in \spa(A,A^+)$ there exists an open neighbourhood $U$ of $x$ in $V(F)$ such that $\overline{\{x\}} \subseteq U$, and an open neighbourhood $V$ of $w(x)$ in $W=\spa(A,A^+)$ such that $\overline{\{w(x)\}}\subseteq V$, $w(U)\subseteq V$ and the induced morphism $w|_U\colon U\rightarrow V $ is finite. By Lemma \ref{4.1} above, the morphism $V(F) \cap \mathcal{D}_n \rightarrow W$ is flat for any $n$. Let $V(F)_n:= (V(F) \cap \mathcal{D}_n)\times_{W} V$. After possibly shrinking $V$, we may assume that $U\subseteq V(F)_n$ for $n$ sufficiently large. Clearly $U \subseteq V(F)_n$ is open. Hence the morphism $U\rightarrow V$ is flat.
\end{proof}
In the proof we have used the following important characterization of locally quasi-finite and partially proper maps, whose proof requires significant effort. 
\begin{prop} \label{1.5.6} Let $f\colon X\rightarrow Y$ be a morphism of analytic adic spaces. Assume $f$ is locally of finite type, $X$ is quasi-separated and $Y$ is quasi-compact and quasi-separated. Then the following are equivalent:
\begin{enumerate}
\item $f$ is locally quasi-finite and partially proper.
\item For every $x\in X$, there exist open subspaces $U$ and $V$ of $X$ and $Y$ such that the closure $\overline{\{x\}} $ of $\{x\}$ in $X$ is contained in $U$, the closure of $\overline{\{f(x)\}}$ of $\{f(x)\}$ in $Y$ is contained in $V$, $f(U)\subseteq V$ and $f\colon U\rightarrow V$ is finite. 
\end{enumerate}
\end{prop}
\begin{proof} See \cite[Proposition 2.16]{johansson}\footnote{By \cite[Lemma 5.1.3]{hu1} $Y$ is taut as it is quasi-compact and quasi-separated.} in this volume, or \cite[Proposition 1.5.6]{hu1}. 
\end{proof}

\begin{rem}\label{conncomp} Consider again the structure map $w\colon V(F)\rightarrow W=\Spa(A,A^+)$ and consider $U$ and $V$ as in Theorem \ref{T.B.1.} above. Then, as noted in the proof of that theorem, for $n$ sufficiently large $U \subseteq V(F)_n = (V(F) \cap \mathcal{D}_n)\times_{W} V$ is open. Now $U\rightarrow V$ is finite, hence also proper and so $U$ is also closed in $V(F)_n $. Therefore $U$ is a union of connected components of $V(F)_n$.
\end{rem}

For the next corollary, we remark the following. Let $Y$ and $Z$ be locally strongly noetherian analytic adic spaces (i.e., both $Y$ and $Z$ have an open cover by affinoids $\Spa(R,R^+)$ with $R$ strongly noetherian and Tate). If $f\colon Y\rightarrow Z$ is a finite flat map, then  
%then the degree of $f$ at a point $y\in Y$ is defined to be the rank of the free $\cO_{Z,f(y)}$-module $\cO_{Y,y}$. The morphism $f$ has constant degree $d$ if 
the sheaf $f_{*}\cO_Y$ is a finite locally free $\cO_{Z}$-module. We say that $f$ has constant degree if $f_{*}\cO_Y$ is of rank $d$ over $\cO_{Z}$. If $Z$ is connected, then any finite flat map is of constant degree. \\

Let $\mathit{Cov}(V(F))$ denote the set of all open affinoid $U\subseteq V(F)$ such that $w(U)\subseteq W$ is open affinoid and the map $w|_U \colon U\rightarrow w(U)$ is finite of constant degree. The above theorem immediately implies:
\begin{corollary} The set $\mathit{Cov}(V(F))$ is an open cover of $V(F)$. 
\end{corollary}
\begin{rem} The analogue of this corollary in the rigid analytic world is \cite[Theorem 4.6]{buz} and its proof is hard and uses Raynaud's theory of admissible formal models of rigid spaces. The hard part here is also to understand the structure of quasi-finite maps. Note that the above proof uses Proposition \ref{1.5.6}, and as mentioned above, the proof of this proposition is also rather difficult.   
\end{rem}

Let us make the geometry of $w$ even more explicit. In the following, we say that two entire power series $S,Q \in  A\{\{X\}\}$ are \emph{relatively prime} if the ideal $(S,Q)$ satisfies $(S,Q) = A\{\{X\}\}$.
\begin{prop}\label{C.B.1} Let $x\in V(F)$ and let $U=\spa(C,C^+)\subset V(F)$ and $V=\spa(B,B^+)\subseteq W$ be affinoid opens as in Theorem \ref{T.B.1.} above. Assume furthermore that $U$ is of constant degree $d$ over $V$. Then the Fredholm series $F$ has a factorization in $B\{\{X\}\}$ of the form $F=QS$, where $Q=1 +b_1X + \dots + b_dX^d  \in B[X]$, $b_d \in B^{\times}$ and $S$ is a Fredholm series with coefficients in $B$ which is relatively prime to $Q$ and $C=B[X]/(Q)$.\\
Conversely, for any such factorization $F=QS$ in $B\{\{X\}\}$, the affinoid
%Assume that we have a factorization of
 %where $Q$ is a polynomial whose leading coefficient is a unit and $S$ is a Fredholm series with coefficients in $B$ which is relatively prime to $Q$. Then
 $$U=\Spa(B[X]/(Q),(B[X]/(Q))^{+}),$$
 is naturally an element of $\mathit{Cov}(V(F))$. (Here $(B[X]/(Q))^{+}$ is the integral closure of $B^+$ in $B[X]/(Q)$.)
\end{prop}
\begin{proof}
For the first part, note that $U$ is a union of connected components of $V(F)_n$ for $n$ sufficiently large, by Remark \ref{conncomp}. Hence we have a surjection 
\[B\langle \varpi^n X\rangle \rightarrow C.\]
Let $L\in B[X]$ be the characteristic polynomial of the image of $X$ in $C$. The composition 
\[f\colon B[X]/(L)\rightarrow B \langle \varpi^nX\rangle/(L) \rightarrow C\]
is a morphism of finitely generated $B$-modules and the topology on $B$ induces a natural topology on them. The topology on $C$ agrees with the topology induced from the surjection $B\langle \varpi^n X \rangle \rightarrow C$. The morphism $f\colon B[X]/(L) \rightarrow C$ therefore has dense image. But by \cite[Lemma 2.3 and 2.4]{hu2}, every morphism of finitely generated $B$-modules is strict, hence $f$ is surjective. Comparing ranks shows that $f$ is an isomorphism. Hence for $n$ large enough, we have a factorization $F=LH$ in $B\langle \varpi^n X \rangle$ with  $H$ a unit in $B\langle \varpi^n X \rangle/(L)\cong C$. From this factorization we see that $L(0) \in B^{\times}$. Define $Q(X):=L(0)^{-1} L(X)$. Then we can rewrite $F=Q\cdot S$ with $S=L(0)H$. As $H$ is a unit in $C$, so is $S$ and hence $Q$ and $S$ are relatively prime. 
 %As $U$ is open and closed in $V(F)_n$, we get a decomposition 
%\[B \langle \vp^{n}X \rangle /(F) \cong B \langle \vp^{n}X \rangle /(Q) \times B \langle \vp^{n}X \rangle /(S) \]

%As $U=V(Q)=\{Q=0\}\subseteq V(F)_n$  and $U$ is open and closed in $V(F)_n$ for $n$ big enough, $S$ does not have any zeroes on $U$ and is therefore invertible on $U$. Hence $Q$ and $S$ are relatively prime. \\
For the ``conversely'' part, we argue as follows: 
$U \rightarrow V$ is finite and surjective of constant degree $\deg Q$, so it remains to see that $U$ is naturally an open subset of $V(F)$. For this we may work locally over $V$. For each $n$ we have compatible morphisms 
\[ B[X]/(Q) \rightarrow  B \langle \vp^{n}X \rangle / (Q) \leftarrow B \langle \vp^{n}X \rangle /(F). \]
The second map is the projection onto the first factor in the decomposition 
\[B \langle \vp^{n}X \rangle /(F) \cong B \langle \vp^{n}X \rangle /(Q) \times B \langle \vp^{n}X \rangle /(S) \] which results from the fact that $Q$ and $S$ are relatively prime. \\
Thus $V(Q)=\{Q=0\}\sub V(F)$ is open and closed in $V(F)$. 
Moreover, when $n$ is sufficiently large, we claim that the first map is an isomorphism. To see this, consider the quotient map $p \colon B[X] \rightarrow B[X]/(Q)$ and equip the target with a submultiplicative norm that induces the canonical topology. 
For large enough $n$ we will have $|p(\varpi^{n}X)|\leq 1$ and hence $p(B_{0}[\varpi^{n}X])\subseteq (B[X]/(Q))_{0}$ (here we are using $-_{0}$ to denote unit balls), so $p$ is continuous for the topology on $B[T]$ coming from the inclusion $B[X]\subseteq B\langle \varpi^{n}X \rangle$, and we may complete to obtain a morphism $B\langle \varpi^{n}X \rangle \rightarrow B[X]/(Q)$ with kernel $QB\langle \varpi^{n}X \rangle$. This gives an inverse to the map $B[X]/(Q) \rightarrow B \langle \varpi^{n}X \rangle / (Q)$, proving the claim.
Thus we may identify $U$ with $\{Q=0\}\subseteq V(F)$, which shows that $U$ is naturally an open subset of $V(F)$.
\end{proof}

\subsection{Slope data}
We end this section with a brief discussion of the concept of a slope datum following \cite[Section 2]{JN1}, which is a helpful terminology in practice and is standard in the context of overconvergent cohomology. 

We keep the general notation of this section. Then for $U=\spa(B,B^+) \subseteq W$ an open affinoid subspace and $h=\frac{l}{m} \in \Q$, we let
\[ \mathcal{D}_{U,h}:= \{x \in \mathbb{A}^1_U: \ |\varpi^lX^m(x)|\leq 1\}\subset \A^1_U\]
similar to above, but keeping track of the base $U$ in the notation. 
\begin{defn}Let $h \in \Q_{\geq 0}$ and let $U\subset W=\spa(A,A^+)$ be an open affinoid. Consider the open affinoid subset  $V_{U,h}:=V(F)\cap \mathcal{D}_{U,h} \subseteq \mathbb{A}^1_W$ of $V(F)$. We say that $(U,h)$ is a \emph{slope datum} (for $(W,F)$) if  $V_{U,h} \rightarrow U$ is finite of constant degree. 
\end{defn}

%das folgende sagt schon fast, dass die spektalvarietät die reziproken der eigenwerte ungleich null parametrisiert. 
Let $h\in \R$. Recall from above that for $A=K$, we say that an entire power series $F\in K\{\{X\}\}$ has slope $\leq h$ (or $>h$) if all slopes of its Newton polygon are $\leq h$ (or $>h$). 
Let us extend this notion Banach--Tate rings as follows. Fix a Banach--Tate ring $A$ and fix a multiplicative pseudo-uniformizer $\vp$. For any rank $1$ point \linebreak $x\in \spa(A,A^+)$ normalize the additive valuation on the complete residue field $k(x)$ so that the image of $\varpi$ under the natural map $A\rightarrow k(x)$ has valuation $1$. 
We say that $F\in A\{\{X\}\}$ has slope $\leq h$ (or $>h$) if, for any rank $1$ point $x\in \spa(A,A^+)$ with residue field $k(x)$, the specialization $F_{x}\in k(x)\{\{X\}\}$ has slope $\leq h$ (or $>h$). 

\begin{defn}
Let $A$ be a Banach--Tate ring with a fixed multiplicative pseudo-uniformizer $\vp$. Let $F\in A\{\{X\}\}$ be a Fredholm series and let $h\in \R$. A \emph{slope $\leq h$-factorization} of $F$ is a factorization $F=QS$ in $A\{\{X\}\}$ where $Q$ is a multiplicative polynomial of slope $\leq h$ and $S$ is a Fredholm series of slope $>h$.
\end{defn}
\begin{rem}\label{slopetop}
If $A$ is a complete Tate ring with a fixed pseudo-uniformizer $\vp$, then the notions of slope factorizations and slope $\leq h$ or $> h$ are independent of the choice of a norm on $A$ with $\vp$ multiplicative. %Moreover, one can define all these notions directly without choosing a norm on $R$.
\end{rem}

\begin{lem}\label{coprime}
Let $A$ be a Banach--Tate ring with a fixed multiplicative pseudo-unifor-mizer $\vp$ and let $h\in \Q_{\geq 0}$. Let $S$ be a Fredholm series of slope $>h$ and $Q$ a multiplicative polynomial of slope $\leq h$. Then $Q$ and $S$ are relatively prime.
\end{lem}
\begin{proof}
See \cite{JN1}, Lemma 2.2.7.
\end{proof}

The following proposition is an adaption Lemma 4.4 of \cite{buz} and Theorem 4.5.1 of \cite{as} to our setting and will be applied in the course of James Newton in the context of spaces of overconvergent $p$-adic forms of bounded slope. 

\begin{prop}\label{constslope} Keep the setup from above, so $W=\spa(A,A^+)$ is as before and \linebreak $F \in A\{\{X\}\}$ is a Fredholm series. 
Let $x$ in $W=\spa(A,A^+)$ be a closed rank $1$ point. Then for any fixed $h \in \Q$, there is a slope datum $(U,h)$ for $(W,F)$ with $U$ an open affinoid neighbourhood of $x$.  
\end{prop}
\begin{proof}
As above %Let $V$ be an open neighbourhood of $x$, containing $\overline{\{x\}}$ and as above 
consider $\mc{D}_{W,h}\subset \A^1_W$. As $w$ is locally quasi-finite, the set 
\[S:=w^{-1}(x) \cap V(F) \cap \mc{D}_{W,h}=\{x_1,\dots, x_n\} \]
is finite. For any $x_i \in S$ let us choose an open neighbourhood $U_i$ as in Theorem \ref{T.B.1.}, so $U_i$ contains $\overline{\{x_i\}}=\{x_i\}$ and we may assume that $w|_{U_i}\colon U_i \rightarrow V_i=w(U_i)$ is finite of constant degree, $U_i$ does not contain any of the $x_j$ for $j\neq i$ and that $V_i$ is open in $W$ and contains $\overline{\{x\}}=\{x\}$. Now (as $x$ is a closed rank $1$ point), for any $i\in \{1,\dots, n\}$
$$\{x\}= \bigcap_{\substack{V_i\supseteq V' \supseteq \overline{\{x\}}\\ V' \text{qc open}}}V'$$
and so also 
$$\{x_i\}= \bigcap_{\substack{V_i\supseteq V' \supseteq \overline{\{x\}}\\ V' \text{qc open}}}w|_{U_i}^{-1}\ (V').$$
Hence by shrinking the affinoids $U_i$, we may assume their images $w(U_i)=:V$ agree, $U_i \rightarrow V$ is finite and that $U_i \subset \mc{D}_{V,h}$ for all $i\in \{1,\dots, n\}$.
We claim that by possibly further shrinking $V$, we can arrange that 
$$w^{-1}(V)\cap \mc{D}_{V,h}= U_1\cup \dots \cup U_n.$$ 
To see this let $w_h:=w|_{w^{-1}(V)\cap \mc{D}_{V,h}}$.
We have $w_h^{-1}(\{x\})\subseteq U_1\cup \dots \cup U_n$. 
Hence there exists $\{x\}\subseteq V' \subset V$, such that $w_h^{-1}(V')\subseteq U_1\cup \dots \cup U_n.$ Redefining $U:=V'$ and letting $U'_i:= w_h^{-1}(U) \cap U_i$ we see that $(U,h)$ is a slope datum for $(W,F)$. 
\end{proof}

\begin{corollary} Let $W=\spa(A,A^+)$ be as above and let $F\in A\{\{X\}\}$ be a Fredholm series. Let $U=\spa(B,B^+) \subseteq W$ be an affinoid open, and $h \in \Q_{\geq0}$. Then the following holds.
\begin{enumerate}
\item If $(U,h)$ is a slope datum then $F$ has a slope $\leq h$-factorization in $B\{\{X\}\}$. 
\item The collection of all $V_{U,h}$ for all slope data $(U,h)$ form an open cover of  $V(F)$. 
\end{enumerate}
\end{corollary}
\begin{proof} We leave it as an exercise for the reader to deduce this from the results above (cf. \cite[Theorem 2.3.2]{JN1} in case you are stuck).  
\end{proof}

\section{Riesz Theory}\label{sec:riesz}
Let $A$ be a noetherian Banach--Tate ring. Let $M$ be a Banach $A$-module satisfying property (Pr) and let $\phi\colon M\rightarrow M$ be a compact endomorphism with characteristic power series $F(X)=\det(1-X\phi)$. In this section we explain how a factorization of $F=Q\cdot S$, where $Q$ is a polynomial and $S$ is an entire power series (e.g. as in Corollary \ref{C.B.1}, or arising from a slope $\leq h$ factorization) corresponds to a (slope) decomposition of $M$. This decomposition is into two direct summands, and crucially (for the eigenvariety construction) one of them is finitely generated.
The main references for this section are \cite[Section 2.2]{JN1}, \cite[Section 3]{buz}, \cite[Appendix B.2]{aip} and \cite[Section 3.2]{bellaiche}.

\medskip

Let $Q\in A[X]$ be a polynomial. We define 
\[Q^{\ast}(X) := X^{\mathrm{deg}Q}Q(1/X) \in A[X].\] 

We need the following lemma. It is proved using the concept of the \emph{resultant} $R(S,Q)$ of an entire power series $S$ and a polynomial $Q$ as developed by Coleman \cite{col} which extends the classical resultant of two polynomials. The resultant $R(S,Q)$ is an element of $A$, which is a unit if and only if $S$ and $Q$ are relatively prime in $A\{\{X\}\}$ (see \cite[Lemma 3.7]{col}).

\begin{lem}\label{3.1}\cite[Lemma A4.1]{col}. Let $A$ be a noetherian Banach--Tate ring. Let $M$ be a Banach $A$-module satisfying property (Pr) and let $\phi\colon M\rightarrow M$ be a compact endomorphism with characteristic power series $S(X)=\det(1-X\phi)$. If $Q(X)\in A[X]$ is a monic polynomial, then $Q$ and $S$ are relatively prime in $A\{\{X\}\}$ if and only if $Q^{\ast}(\phi)$ is an invertible operator on $M$. 
\end{lem}
Let $M$ and $\phi$ be as above. The above mentioned decomposition of $M$ is constructed in two steps. First we use zeroes of the characteristic power series $F(X)=\det(1-X\phi)$ to split off a finite submodule of $M$ (see Proposition \ref{3.2}). In a second step we prove the sought after decomposition, in what is sometimes referred to as the Main Theorem of Riesz theory, Theorem \ref{riesz}.

For the first step, let us define the \emph{Fredholm resolvant} of $\phi$ to be 
$$R_F(X):=F(X)/(1-X\phi)\in A[\phi]\llbracket X\rrbracket.$$ One shows as in \cite[Proposition 10]{serre} that if one writes $R_F(X)=\sum_{m\geq 0} v_mX^m$, with coefficients $v_m \in \End_A(M)$, then for all $r\in \mathbb{R}_{>0}$ the sequence $||v_m||r^m$ tends to zero.

Let us also make the following remark regarding zeroes of power series. 
Let $R$ be a ring. If $f=\sum_{n\geq0} a_n X^n \in R\llbracket X\rrbracket$ and $s\in \mathbb{N}_{0}$, then we define
\[\Delta^s f = \sum_{n\geq 0} {n+s \choose s}a_{n+s} X^n \in R\llbracket X\rrbracket.\]
If $f,g \in R\llbracket X\rrbracket$, then one can check that $\Delta^s(fg)= \sum_{i=0}^s \Delta^{i}(f) \Delta^{s-i}(g)$. Let $A$ be as before. It is clear from the definition, that $\Delta^s$ maps $A\{\{X\}\}$ to itself. 
\begin{defn}
Let $k \in \mathbb{N}_{0}$. We say that $a \in A$ is a \emph{zero of order} $k$ of $H\in A \{\{X\}\}$ if $(\Delta^sH)(a)=0$ for $s<k$ and $(\Delta^kH)(a)$ is a unit. 
\end{defn}
If $k\geq 1$ and $H= 1+ a_1X + \dots$ then this implies that $-1=a(a_1+ a_2a + \dots)$ and hence that $a\in A$ is a unit. By induction on $k $ one 
checks that $H(X)=(1-a^{-1}X)^kG(X)$, where $G\in A\{\{X\}\}$ with $G(a)$ a unit. 

\begin{prop}\label{3.2}
Let $M$ be a Banach $A$-module with property $(Pr)$, let $\phi\colon M\rightarrow M$ be a compact endomorphism with characteristic power series $F(X)=\sum_{n\geq 0} c_nX^n$. Let $a\in A$ be a zero of $F(X)$ of order $k$. 
There is a unique decomposition $M= N\oplus M'$ into closed $\phi$-stable submodules such that $1-a\phi$ is invertible on $M'$ and $(1-a\phi)^k=0$ on $N$. The submodules $N$ and $M'$ are defined as the kernel and the image of a projector which is in the closure in $\End_A(M)$ of $A[\phi]$. Moreover $N$ is projective of rank $k$, and assuming $k\geq 1$ then $a $ is a unit and the characteristic power series of $\phi$ on $N$ is $(1-a^{-1}X)^k$. 
\end{prop}
\begin{proof} The proof of Serre (\cite{serre}) and Buzzard (\cite{buz}) goes through: 
In $A[\phi]\llbracket X\rrbracket$, we consider the identity
\[(1-X\phi)\cdot R_F(X)= F(X).\]
Applying the operator $\Delta^s$ we find 
\[(1-X\phi)\Delta^sR_F(X)-\phi \Delta^{s-1}R_F(X)=\Delta^sF(X). \]
Define $f_s:=\Delta^sR_F(a)$. Putting $X=a$ in the identities above, we arrive at the following relations:
\begin{align*}
(1-a\phi)f_0 &= 0\\
(1-a\phi)f_1-\phi f_0 &= 0\\
\cdots \\
(1-a \phi) f_{k-1} - \phi f_{k-2}&=0\\
(1-a \phi) f_{k} - \phi f_{k-1}&=c,\ \text{with } c \in A^{\times}. 
\end{align*}
%with $c \in A^{\times}$.
These relations imply that for $s<k$, we have $(1-a\phi)^{s+1}f_s=0$. The last equation and the fact that $(1-a\phi)^kf_{k-1}=0$ imply that with $e=c^{-1}(1-a\phi)f_k$ and $f=- c^{-1} \phi f_{k-1}$ we get 
\[e+f=1 \ \ \text{ and } fe^k=0.\]
By expanding $(e+f)^k=1$, we find that $e^k + (k e^{k-1}f + \dots + k ef^{k-1}+ f^k)=1$. 
We define 
\[p:= e^k, q:= k e^{k-1}f + \dots + k ef^{k-1}+ f^k.\]
Then $p+q=1$, $pq=0$ and furthermore $p^2=p$ and $q^2=q$. The endomorphisms $p$ and $q$ in $\End_A(M)$ are therefore projections. Note that they are both in the closure of $A[\phi]$.\\

Consider the decomposition $M=N \oplus M'$ corresponding to these projections, i.e. $N=\mathrm{Ker}(p)$, and $M'=\mathrm{Ker}(q)$. Then $(1-a\phi)^k=0$ on $N$ and $(1-a\phi)$ is invertible on $M'$. Furthermore if we set $\psi:= (1-a\phi)^k$, then $N=\mathrm{Ker}(\psi)$ and $F=\mathrm{Im}(\psi)$. \\
It is clear that $N$ satisfies property $(Pr)$, but furthermore we have $(1-a\phi)^k=0$ on $N$, which implies that the identity is a compact operator on $N$. Exactly as in the proof of \cite[Proposition 3.2]{buz} one shows from this that $N$ is projective. 
%Changing the norm on $N$ to an equivalent one if necessary and reducing to a computation of matrices shows the following: If $\beta \in \End_A(N)$ has sufficiently small norm, then $||\beta^n|| \rightarrow 0 $ and hence $1-\beta$ is invertible. As $1$ is compact, we can choose $\alpha \colon N\rightarrow N$ of finite rank such that $1-\alpha$ has norm sufficiently small and hence $\alpha $ is invertible on $N$ and therefore $N$ is finitely-generated. By Exercise \ref{proj}(2), $N$ is projective. \\

If $k=0$, then $N=0$ and $M'=M$ as can be seen from Lemma \ref{3.1} and we are done. If $k\geq 1$ then $F(a)=0$ and we already observed that then $a$ is a unit. Let $F_N$ and $F_{M'}$ denote the characteristic power series of $\phi$ on $N$ and $M'$ respectively, then $F=F_NF_{M'}$. By Lemma \ref{3.1} we see that $(X-a)^k$ and $F_{M'}$ generate the unit ideal in $A\{\{X\}\}$. Hence $(1-a^{-1}X)^k$ divides $F_N$ in $A\{\{X\}\}$. We leave it as an exercise to show that in fact we have equality $F_N=(1-a^{-1}X)^k$ (cf.\cite[Prop.\ 3.2]{buz}).
% But $F_N$ is a polynomial, as $N$ is finitely generated and hence $(1-a^{-1}X)^k$ divides $F_N$ in $A[X]$. Furthermore $(1-a^{-1}X)^k$ has constant term $1$ and hence is not a zero-divisor in $A\{\{X\}\}$, hence if we write $F_N(X)=(1-a^{-1}X)^k D(X)$ and $F(X)=(1-a^{-1}X)^k G(X)$, then $D(X)$ divides $G(X)$ in $A\{\{X\}\}$ and hence $D(a)$ is a unit. \\
%We know that $D(X)$ is a polynomial. Furthermore, as $(1-a\phi)^k=0$ on $N$, we see that $\phi$ has an inverse on $N$ and hence the determinant of $\phi$ is a unit in $A^{\times}$. Hence the leading term of $D$ is a unit. Reducing the situation modulo a maximal ideal of $A$ we see that the reduction of $F_N$ must be a power of the reduction of $(1-a^{-1}X)$ and from this we can conclude that $D=1$. Hence the characteristic power series of $\phi$ on $N$ is $(1-a^{-1}X)^k$.\\

Finally, as $(1-a\phi)^k=0$ on $N$, we see that $\phi$ is invertible on $N$. This means that the rank of $N$ at any maximal ideal must be equal to the degree of $F_N$ modulo this maximal ideal and hence the rank is $k$ everywhere. 
\end{proof}
\begin{rem}\label{eigen1}
Note that (in the notation of the proposition) in particular the zeroes of $F$ are the inverses of the eigenvalues of $\phi$ and all eigenvectors for the eigenvalue $a^{-1}$ are contained in $N$. 
\end{rem}

\begin{defn} We call a polynomial $Q\in A[X]$ \emph{multiplicative}, if its leading coefficient is a unit, i.e., if $Q^*(0)\in A^{\times}$.  
\end{defn}

\begin{thm}\label{riesz}
Let $A$ be a noetherian Banach--Tate ring. Let $M$ be a Banach $A$-module with property (Pr) and let $\phi \colon M \rightarrow M$ be a compact operator with characteristic power series  $F(X)=\det(1-X\phi)$. Assume that we have a factorization $F=QS$, where $S$ is a Fredholm series, $Q=1+\dots + a_nX^n \in A[X]$ is a polynomial of degree $n$ which is multiplicative, and $Q$ and $S$ are relatively prime in $A\{\{X\}\}$. 
\begin{enumerate}
\item The submodule $\Ker Q^{\ast}(\phi)\subseteq M$ is finitely generated and projective and has a unique $\phi$-stable closed complement $M(Q)$ such that $Q^{\ast}(\phi)$ is invertible on $M(Q)$. 
\item The idempotent projectors $M \rightarrow \Ker Q^{\ast}(\phi)$ and $M \ra M(Q)$ lie in the closure of $A[\phi]\subseteq \End_{A}(M)$. 
\item The rank of $\Ker Q^{\ast}(\phi)$ is $\deg Q$, and $\det(1-X\phi \mid \Ker Q^{\ast}(\phi))=Q$. 
\item Moreover, $\phi$ is invertible on $\Ker Q^{\ast}(\phi)$ and $\det(1-X\phi\mid M(Q))=S$.
\end{enumerate}
\end{thm}
\begin{proof}
Consider the operator $v:= 1-Q^{\ast}(\phi)/Q^{\ast}(0)$. This is compact and has a characteristic power series which has a zero at $X=1$ of order $n$ (see \cite[A4.3]{col}). Applying the previous proposition to $v$, we see that $M=N\oplus M(Q)$, where $N$ and $M(Q)$ are defined as the kernel and the image of a projector in the closure of $A[v]$ and hence in the closure of $A[\phi]$. Both submodules $N$ and $M(Q)$ are $\phi$-stable. %Note that $N=\mathrm{Ker}(1-v)^n= 
Furthermore $N$ is projective of rank $n$. Moreover by Lemma \ref{3.1} the characteristic power series of $\phi$ on $M(Q)$ is coprime to $Q$. Hence if $G(X)$ is the characteristic power series of $\phi$ on $N$, we see that $Q$ divides $G$. But $Q$ and $G$ both have degree $n$, and the same constant term. Furthermore the leading coefficient of $Q$ is a unit. This is enough to deduce that $G=Q$. We conclude that $N=\mathrm{Ker}(Q^{\ast}(\phi))$ and we have thus shown the first three parts. \\
We prove the last part as follows. Note that $\det(\phi \mid \Ker Q^{\ast}(\phi))=Q^{\ast}(0)\in A^{\times}$ and so $\phi$ is invertible on $\Ker Q^{\ast}(\phi)$. For the last claim write $S^{\prime}=\det(1-X \phi \mid M(Q))$ and note that $F=\det(1-X\phi \mid \Ker Q^{\ast}(\phi))\det(1-X\phi \mid M(Q))=QS^{\prime}$. Hence $Q(S-S^{\prime})=0$. But  $Q$ is not a zero divisor since $Q(0)=1$ and  so $S=S^{\prime}$.
\end{proof}
In the rest of this section we include without proofs a discussion of \emph{slope $\leq h$-decompositions} from \cite{JN1} (cf.\cite{as}). Again this is helpful terminology in the context of overconvergent cohomology. 

\begin{defn}
Let $M$ be an (abstract) $A$-module, let $\phi\colon M \ra M$ be an $A$-linear map and let $h\in \Q$. An element $m\in M$ is said to have slope $\leq h$ with respect to $\phi$ if there is a multiplicative polynomial $Q\in A[X]$ such that
\begin{enumerate}
\item $Q^{\ast}(\phi)(m)=0$,
\item The slope of $Q$ is $\leq h$.
\end{enumerate}
We let $M_{\leq h}\sub M$ denote the subset of elements of slope $\leq h$.
\end{defn}

\begin{lem}
$M_{\leq h}$ is an $A$-submodule of $M$, which is stable under $\phi$.
\end{lem}

\begin{proof} See \cite[Lemma 2.2.9]{JN1}.
%It is clear from the definition that $M_{\leq h}$ is closed under multiplication, and stable under $u$. It therefore suffices to prove that it is closed under addition, for which it suffices to prove that if $Q_{1}$ and $Q_{2}$ are two multiplicative polynomials of slope $\leq h$, then so is $Q_{1}Q_{2}$. To see this it suffices to specialize to the case when $R$ is a field and the norm is an absolute value. The assertion is then well known (for example, the argument in the proof of  \cite[Proposition 4.6.2]{as} carries over without change).
\end{proof}

\begin{defn}%(\cite[Definition 4.6.3]{as})
Let $M$ be an $A$-module with an $A$-linear map $\phi \colon M \ra M$ and let $h\in \Q$. A slope $\leq h$-decomposition of $M$ is an $A[\phi]$-module decomposition \linebreak $M=M_{h}\oplus M^{h}$ such that
\begin{enumerate}
\item $M_{h}$ is a finitely generated $A$-submodule of $M_{\leq h}$,
 \item For every multiplicative polynomial $Q\in A[X]$ of slope $\leq h$, the map $$Q^{\ast}(\phi)\colon M^{h} \ra M^{h}$$ is an isomorphism of $A$-modules.
\end{enumerate}
\end{defn}

\begin{prop}\cite[Proposition 2.2.11]{JN1}. Let $M$ and $\phi$ be as above. If $M$ has a slope $\leq h$-decomposition $M_{h}\oplus M^{h}$, then it is unique, and $M_{h}=M_{\leq h}$ (in particular the latter is finitely generated over $A$). We write $M_{>h}$ for the unique complement. Moreover, slope decompositions satisfy the following functorial properties:
\begin{enumerate} 
\item Let $f\colon M \ra N$ be a morphism of $A[\phi]$-modules with slope $\leq h$-decompositions. Then $f(M_{\leq h})\sub N_{\leq h}$ and $f(M_{>h})\sub N_{>h}$. Moreover, both ${\rm Ker}(f)$ and ${\rm Im}(f)$ have slope $\leq h$-decompositions.

\item Let $C^{\bullet}$ be a complex of $A[\phi]$-modules and suppose that each $C^{i}$ has a slope $\leq h$-decomposition. Then every $H^{i}(C^{\bullet})$ has a slope $\leq h$-decomposition, explicitly given by $H^{i}(C^{\bullet})=H^{i}(C^{\bullet}_{\leq h}) \oplus H^{i}(C^{\bullet}_{>h})$.
\end{enumerate}
\end{prop}
\begin{comment}
\begin{proof}
The proof is identical to that of \cite[Lemma 4.6.4]{as}: one equates slope $\leq h$-decompositions with $\mc{S}$-decompositions (as defined and studied in \cite[\S 4.1]{as})) for the set $\mc{S}\sub R[u]$ of all $Q^{\ast}(u)$ where $Q$ is a multiplicative polynomial of slope $\leq h$. The properties then stated follow from general facts about $\mc{S}$-decompositions, recorded in \cite[Proposition 4.1.2]{as}.
\end{proof}
\end{comment}

\begin{defn}
Let $A$ be a Banach--Tate ring with a fixed multiplicative pseudo-uniformizer $\vp$ and let $M$ be a Banach $A$-module. Assume that $M$ has a slope-$\leq h$ decomposition $M=M_{\leq h}\oplus M_{>h}$. If $f \colon A\ra B$ is a bounded morphism of Banach--Tate rings such that $f(\vp)$ is a multiplicative pseudo-uniformizer in $B$, we say that the slope-$\leq h$ decomposition is \emph{functorial for $f$} if the decomposition 
$$M\widehat{\otimes}_{A}B = (M_{\leq h} \otimes_{A}B) \oplus (M_{>h} \widehat{\otimes}_{A}B)$$
 is a slope-$\leq h$ decomposition of $M \widehat{\otimes}_{A}B$ (using $f(\vp)$ to define slopes for $B$). We say that the slope-$\leq h$ decomposition is \emph{functorial} if it is functorial for all such bounded homomorphisms of Banach--Tate rings out of $A$.
\end{defn}

\begin{thm}\cite[Theorem 4.12]{JN1}.
Let $(A,A^+)$ be a noetherian Tate-Huber pair with a fixed multiplicative pseudo-uniformizer $\vp$, and let $M$ be a Banach $A$-module with property (Pr). Let $\phi$ be a compact $A$-linear operator on $M$, with Fredholm determinant $F=\det(1-X\phi)$. If $M$ has a slope $\leq h$-decomposition which is functorial with respect to $A \ra k(x)$ for all rank $1$ points $x\in \spa(A,A^+)$, then $F$ has a slope $\leq h$-factorization. Conversely, if $F$ has a slope-$\leq h$ factorization, then $M$ has a functorial slope-$\leq h$ decomposition.
\end{thm}

\begin{rem}
From the theorem and the fact that slopes of Newton polygons of polynomials correspond to valuations of the roots, we see the following: If $M$ has a slope $\leq h$ decomposition and $m_x \in M_{\leq h}\otimes_{A}k(x)$ is an eigenvector for $\phi$ with eigenvalue $\lambda_x\in k(x)$, then $\lambda_x$ has slope $\leq h$, i.e., $v_{\varpi}(\lambda_x)\leq h$. Here $v_{\varpi}$ is the $\varpi$-adic valuation induced by $\varpi$ (i.e., $v(\varpi)=1$). 
\end{rem}
\begin{corollary}
With notation and assumptions as in the theorem, a slope-$\leq h$ decomposition of $M$ is functorial if and only if it is functorial for the natural map $A \ra k(x)$ for all rank $1$ points $x\in \spa(A, A^+)$. 
\end{corollary}

\newpage
\section{The Eigenvariety Machine}\label{sec:evs}
We are ready to study the \emph{eigenvariety machine}. In a first step we will learn how to construct an \emph{eigenvariety} over an affinoid base. In a second step we will study the general machine that allows for more general bases and a flexible general setup. 
For this section we also fix a non-archimedean field $K$ and assume it is discretely valued. We let $\mc{O}_K$ be its ring of integers and $\pi_K$ a uniformizer. The main references are \cite{buz} and \cite{JN2}.

\subsection{Eigenvarieties over an affinoid base}
Let  $(A,A^+)$ be a noetherian Tate--Huber pair, with $A$  a Banach--Tate ring with pseudo-uniformizer $\varpi$, equipped with the standard norm. We also assume that $A$ is an \linebreak $\cO_K$-algebra. Let $W=\spa(A,A^+)$ and
let $M$ be a Banach $A$-module satisfying property (Pr). Let $\mathbf{T}$ be a commutative $\cO_K$-algebra equipped with an $\cO_K$-algebra homomorphism to $\mathrm{End}_A(M)$, the continuous $A$-module endomorphisms of $M$.  In practice  $\mathbf{T}$ will be a polynomial $\cO_K$-algebra generated by typically infinitely many Hecke operators. In the notation we  often do not distinguish an element $t\in \T$ from the endomorphism of $M$ associated to it. 

Fix once and for all an element $\phi \in \T$ and assume the induced endomorphism $\phi\colon M\rightarrow M $ is compact. Let $F(X)= 1+ \sum_{n\geq1} c_nX^n = \det(1-X\phi)$ be the characteristic power series of $\phi$. Let $\cZ:=V(F) \subset \mathbb{A}^1_W$ be the spectral variety of $F$.

Let us call a tuple $((A,A^+), M , \T, \phi)$ of such objects a \emph{spectral datum}. 
Our main goal now is to construct from such a spectral datum the so called \emph{eigenvariety}, which is a certain finite cover of $\cZ$.  %Eventually we'll see that (under additional assumptions on $W$) a point on the eigenvariety will correspond to a system of eigenvalues for all the operators in $\T$, such that the eigenvalue for $\phi$ is non-zero. 
%Then we will learn how to glue this construction to allow more general (not necessarily affinoid) bases $W$. 
\medskip

Before we give the construction and as in \cite{buz} let us discuss an instructive finite-dimensional example, where $\phi$ is invertible.
For that assume $(A,A^+)$ is as above and let $M$ be a finitely generated projective $A$-module of rank $d$. Let $\T$ be a commutative $A$-algebra equipped with an $A$-algebra homomorphism to $\mathrm{End}_A(M)$ and let $\phi$ be an element of $\T$. Assume that the image of $\phi$ in  $\mathrm{End}_A(M)$ lies in $\mathrm{Aut}_A(M)$.\\ %, i.e., there exists $\phi^{-1}\colon M\rightarrow M$ such that $\phi\circ \phi^{-1}= \phi^{-1}\circ\phi= \mathrm{id}_M$. \\
Define $P(X)=\det(1-X\phi)=1+\dots \in A[X]$. Then the leading term of $P$, i.e., the coefficient of $T^d$ is a unit and $A[X]/(P(X))$ is a finite $A$-algebra. 
Let \[\mathcal{Z}:=\spa(A[X]/(P(X)),A[X]/(P(X))^+) ,\] where $A[X]/(P(X))^+$ is the integral closure of $A^+$ in $A[X]/(P(X))$. \\
Let $\T(\cZ)$ be the image of $\T$ in $\mathrm{End}_A(M)$, then $\T(\cZ)$ is a finite $A$-algebra. The Cayley--Hamilton theorem implies that $\phi^{-1}\in \T(\cZ)$ and that there is a natural map $A[X]/(P(X))\rightarrow \T(\cZ)$, sending $X$ to $\phi^{-1}$.
Define $\cE:= \spa(\T(\cZ), \T(\cZ)^+)$, where $\T(\cZ)^+$ is the integral closure of $A^+$ in $\T(\cZ)$. The maps $A\rightarrow A[X]/(P(X))\rightarrow \T(\cZ)$ respect the rings of integral elements and give rise to maps
\[\cE \rightarrow \cZ \rightarrow \spa(A,A^+).\]
The space $\cE$ is called the eigenvariety associated to the datum $((A,A^+),M, \T, \phi)$. 
\begin{ex} Consider the case, where $A=K\langle T_1, T_2\rangle$, for some non-archimedean field $K$, where $M=A^2$ and $\T= A[\phi, t]$, where $\phi$ acts on $M$ as the matrix $
\left(^1_0 \ ^{T_1}_1\right)$ and $t$ acts as $\left(^0_0\ ^{T_2} _0\right)$. Then $\det(1-X\phi)=(1-X)^2$, so in particular $\cZ$ is non-reduced, and $\T(\cZ)$ is the ring $A\oplus I\varepsilon $ with $I=(T_1,T_2)$ and $\varepsilon^2=0$. Note that in this case the maps $\cE\rightarrow \cZ$ and $\cE\rightarrow \spa(A,A^+)$ are not flat and $\cE$ is not reduced. 
\end{ex}
\begin{rem} So in general one cannot hope for too many good geometric properties of an eigenvariety. However, those occurring in nature are sometimes better behaved. 
\end{rem}

Let us now go back to the more general setup of a spectral datum \linebreak $((A,A^+), M , \T, \phi)$ (as fixed at the beginning of this section). Functional analysis and our study of the spectral variety allow us to construct the eigenvariety in this more general setup as follows.
\medskip

\noindent \textbf{Construction}: Let $U\in \mathit{Cov}(\cZ)$, so $U\subset \cZ$ is open affinoid with image $V=w(U)=\spa(B,B^+)\subseteq W$ affinoid open, and $U\rightarrow V$ is finite and flat. Define $M_B:=M\widehat{\otimes}_A B$ and for $t\in \T$ let $t_U$ denote the $B$-linear endomorphism of $M_B$ induced by $t\colon M\rightarrow M$. Note that by Corollary \ref{bccomp}, $\phi_U\colon M_B\rightarrow M_B$ is still compact. Let $F_U(X)$ be the characteristic power series of $\phi_U$ on $M_B$. By Corollary \ref{bcchar}, $F_U$ is just the image of $F$ in $B\{\{X\}\}$. 

By Corollary \ref{C.B.1}, we know that $\cO_{\cZ}(U)=B[X]/(Q)$ where $Q=1+b_1X +\dots + b_dX^d \in B[X]$ is a polynomial with $b_d\in B^\times$ and we have a factorization $F_U=QS$, where $S$ is a Fredholm series in $B\{\{X\}\}$ which is relatively prime to $Q$. 
We can now invoke Theorem \ref{riesz} to get a decomposition 
$$M_B=N\oplus M(Q),$$ where $N=\mathrm{Ker}Q^*(\phi_U)$ is projective of rank $d$ over $B$. As the projector $M\rightarrow N$ is in the closure of $B[\phi_U]$ it commutes with all the endomorphisms of $M_B$ induced by the elements of $\T$, hence $N$ is $t$-invariant for all $t\in \T$. 

Define $\T_U$ to be the $B$-subalgebra of $\mathrm{End}_B(N)$ generated by the elements of $\T$. Now $\mathrm{End}_B(N)$ is a finite $B$-module, and hence $\T_U$ is a finite $B$-algebra and hence a Tate-Huber ring. Let $\T_U^+$ be the integral closure of $B^+$ in $\T_U$. Define 
$$\cE_U:=\spa(\T_U, \T_U^+).$$
We know that $Q^*(\phi_U)$ is zero on $N$ and hence $\T_U$ is naturally a finite $B[Y]/(Q^*(Y))$-algebra. Moreover, as the constant term of $Q^*$ is a unit, there is a natural isomorphism $B[Y]/(Q^*(Y))\cong B[X]/(Q(X))$ sending $Y$ to $X^{-1}$. Hence $\T_U$
is a finite $B[X]/(Q(X))$-algebra and the map $B[X]/(Q(X))\rightarrow \T_U$ sends $X$ to $\phi^{-1}$. As $\T_U^+$ is the integral closure of $B^+$ in $\T_U$, and similarly for $B[X]/(Q(X))^+$, we see that $\T_U^+$ also agrees with the integral closure of 
$B[X]/(Q(X))^+$ in $\T_U$. Hence we have a finite map
$\cE_U\rightarrow U$. The space $\cE_U$ is the local piece above $U$ of the eigenvariety that we seek to construct. 

We now proceed by showing that these local pieces glue together when $U$ ranges through the elements of $\mathit{Cov}(\cZ)$.
We do this by first checking that the projective modules $N$ that we get from Riesz theory, glue together to form a coherent sheaf on $\mc{Z}$. In the following let us write $F_U=Q_US_U$ for the  factorization associated to an element $U\in \mathit{Cov}(\cZ)$. 
\begin{prop}\label{5.3} The assignment $U\mapsto \mathrm{Ker}(Q_U^*(\phi_U))$ defines a coherent sheaf $\mathcal{N}$ of $\cO_{\cZ}$-modules on $\cZ$.
\end{prop}
\begin{proof} We include the proof from \cite[4.1.8]{JN1}.
As $\mathit{Cov}(\cZ)$ is an open cover of $\cZ$ it suffices to prove that whenever $U_{1}\sub U_{2}$ are elements of $\mathit{Cov}(\cZ)$, we have a canonical isomorphism 
\[\mathrm{Ker}(Q^{*}_{U_2}(\phi_{U_2}))\otimes_{\cO_{U_2}} \cO_{U_1} \cong \mathrm{Ker}(Q^{*}_{U_1}(\phi_{U_1})). \]
Define $U_{3}=w|_{U_{2}}^{-1}(w(U_{1}))$. Then $U_{1}\sub U_{3} \sub U_{2}$, so it suffices to treat the inclusions $U_{1}\sub U_{3}$ and $U_{3}\sub U_{2}$. In the first case we have $w(U_{1})=w(U_{3})$ and the result follows from Proposition \ref{vertical} below, since $U_{1}\sub U_{3}$ forces $Q_{U_{1}}|Q_{U_{3}}$. For the second we have $U_{3}= U_{2} \times_{w(U_{2})}w(U_{1})$ and we leave it as an exercise to check that everything behaves nicely. (Note that we already know from Lemma \ref{specbc} that the spectral varieties behave naturally with respect to base change.)
\end{proof}

\begin{prop}
\label{vertical}
Assume that $F$ has two factorizations $F=Q_{1}S_{1}=Q_{2}S_{2}$ where the $Q_{i}$ are multiplicative polynomials, the $S_{i}$ are Fredholm series, and for each $i$ the $Q_{i}$ and $S_{i}$ are relatively prime. Assume furthermore that $Q_{1}|Q_{2}$. Then we have a canonical isomorphism 
$$\mathrm{Ker}(Q_{2}^{\ast}(\phi)) \otimes_{A[X]/(Q_{2})} A[X]/(Q_{1}) \cong \mathrm{Ker}(Q_{1}^{\ast}(\phi)).$$
\end{prop}
\begin{proof}
Let $P\in A[X]$ be such that $Q_{2}=PQ_{1}$. Then $P$ is a multiplicative polynomial which is relatively prime to $Q_{1}$. Hence we may find polynomials $f,g\in A[X]$ such that $Pf+Q_{1}g=1$. We then have $A[X]/(Q_{2}) \cong A[X]/(Q_{1}) \times A[X]/(P)$ and $Pf\in A[X]/(Q_{2})$ corresponds to $(1,0)\in A[X]/(Q_{1}) \times A[X]/(P)$. 
The proposition follows from the equality
$$Pf\cdot \mathrm{Ker}(Q_{2}^{\ast}(\phi))=\mathrm{Ker}(Q_{1}^{\ast}(\phi)),$$ which we leave as an exercise (cf.\ \cite[Proposition 4.1.7]{JN1}). 
%If $x\in \Ker^{\bu}Q_{2}^{\ast}(\wt{U}_{\ka})$, then 
%$$ Q^{\ast}_{1}(\wt{U}_{\ka}).PAx= \wt{U}_{\ka}^{\deg Q_{1}}AQ_2x=0,$$
%which gives us one inclusion. For the other, assume that $y\in \Ker^{\bu} Q_{1}^{\ast}(\wt{U}_{\ka})$. Note that $\deg PA =\deg Q_{1}B$. Then 
%$$ y = \wt{U}_{\ka}^{-\deg PA}(P^{\ast}(\wt{U}_{\ka})A^{\ast}(\wt{U}_{\ka})+B^{\ast}(\wt{U}_{\ka})Q_{1}^{\ast}(\wt{U}_{\ka}))y = P^{\ast}(\wt{U}_{\ka})A^{\ast}(\wt{U}_{\ka})\wt{U}_{\ka}^{-\deg PA}y, $$
%which gives us the other inclusion.
\end{proof}

We can now define the eigenvariety. 
Let $((A,A^+),M, \T, \phi)$ be a spectral datum. Let $\mathcal{N}$ be the coherent sheaf on the spectral variety $\cZ$ from  above. Consider the endomorphism sheaf $\mathrm{End}_{\cO_{\cZ}}(\mathcal{N})$. For every open $U\in \mathit{Cov}(\cZ)$ we have a map $\T\rightarrow \End_{\cO_{\cZ}}(\mathcal{N})(U)$ and these maps are compatible. Define $\T(\cZ)$ to be the sub-presheaf of $\End_{\cO_{\cZ}}(\mathcal{N})$ generated over $\cO_{\cZ}$ by the image of $\T$. By flatness of rational localization this is a sheaf\footnote{See the proof of Lemma \ref{flatbc} below for the reason why flatness is crucial.} and hence a coherent sheaf of $\cO_{\cZ}$-algebras on $\cZ$.

\begin{defn}
The \emph{eigenvariety} $\cE$ associated to the spectral datum $((A,A^+),M, \T, \phi)$ is defined as the relative adic spectrum $\underline{\spa}(\T(\cZ))$. 
\end{defn}
By construction the eigenvariety $\cE$ is finite over $\cZ$. It is quasi-separated and locally quasi-finite and partially proper over $W=\spa(A,A^+)$, as the structure morphism $w\colon \cZ\rightarrow W$ is. 
\bigskip

Let us check what happens when we change the base to a different affinoid algebra. First a lemma from commutative algebra.
\begin{lem}\label{nilpotence}\cite[Lemma 3.1.2]{JN2}
Let $A$, $A'$ and $T$ be rings with $A$ and $A'$ Noetherian and let $N$ be a finitely generated $A$-module. Assume that we have ring homomorphisms $f \colon A \ra A'$ and $\psi_{A} \colon T \ra \End_{A}(N)$ and let $\psi_{A'}$ be the composition of $\psi_{A}$ with the natural map $\End_{A}(N) \ra \End_{A'}(N\otimes_{A}A')$ coming from $f$. Let $T_{A}$ (resp. $T_{A'}$) be the $A$-subalgebra (resp. $A'$-subalgebra) generated by the image of $\psi_{A}$ (resp. $\psi_{A'}$). Then the natural \linebreak $A'$-linear map $T_{A}\otimes_{A}A' \ra T_{A'}$ is a surjection with nilpotent kernel in general, and if $f$ is flat then it is an isomorphism.
\end{lem}

Now let us consider a flat map $(A,A^+) \rightarrow (A', {A'}^+)$ of noetherian Tate-Huber pairs (note this is adic). Let $M'$ and $\phi'$ be the obvious base extensions, so e.g. $M':= M\widehat{\otimes}_A A'$. Consider the spectral datum $((A', {A'}^+),M',\T, \phi')$. To it we can associate a spectral variety $\cZ'$ and an eigenvariety $\cE'$. We already know from Lemma \ref{specbc} that $\cZ'=\cZ\times_{\spa(A,A^+)} \spa(A', {A'}^+)$. 
\begin{lem}\label{flatbc}
If $(A,A^+)\rightarrow (A',{A'}^+)$ is a morphism of noetherian Tate--Huber pairs, then there is a natural map $\mc{E}' \ra \mc{E}$ over $\mc{Z}$ and the induced map $\cE' \rightarrow \cE\times_{\mc{Z}} \mc{Z'}$ induces an isomorphism $(\cE')^{red} \rightarrow (\cE\times_{\mc{Z}} \mc{Z'})^{red}$.
If the map $A\rightarrow A'$ is flat, then $\cE' \rightarrow \cE\times_{\mc{Z}} \mc{Z'}$ is an isomorphism.
\end{lem}
\begin{proof} Let $\mathit{Cov}(\cZ)$ and $\mathit{Cov}(\cZ')$ denote the usual sets of covers of $\cZ$ and $\cZ'$. Let $W=\spa(A,A^+)$ and $W'= \spa(A', {A'}^+)$. If $U$ is in $\mathit{Cov}(\cZ)$, then the pullback $U'=U\times_{W} W'$  is an element of $\mathit{Cov}(\cZ')$ and the $U'$ where $U$ runs through $\mathit{Cov}(\cZ)$ still form a cover of $\cZ'$. Hence $\cE'$ may be constructed by glueing the $\cE_{U'}$ for all such $U'$. So let $U \in \mathit{Cov}(\cZ)$ with image $V=\spa(B,B^+)$ in $W$. Applying Lemma \ref{nilpotence} the statements follow in this case. In particular if $A\rightarrow A'$ is flat, so is $B\rightarrow B':=B\widehat{\otimes}_A A'$ and we have $\T_{U'}= \T_U\otimes_B B'$. As the assertion is local on both $\mc{Z}$ and $\mc{Z}'$ the lemma follows. 
\end{proof}

\subsection{Eigenvariety data}
In practice the base of the eigenvariety is not affinoid. 
To fix notation let us take an abstract locally noetherian analytic adic space $\mc{W}$ over $\mc{O}_K$ as a base space. The letter $\mc{W}$ is used as in practice this space is a space of \emph{weights} of $p$-adic automorphic forms. The base $\mc{W}$ is then called the weight space, we adapt this terminology here (even though it is content-free). For example, in the case of the Coleman--Mazur eigencurve, the base $\mc{W}$ is a finite disjoint union of open discs. Then over an open affinoid $\spa(A,A^+)$ of $\mc{W}$ one typically has a situation as described above, i.e., a spectral datum $((A,A^+),M,\T, \phi)$, and one would like to glue the respective eigenvarieties to an eigenvariety over $\mathcal{W}$. The algebra $\T$ is usually some abstract algebra that is  independent of $A$, but unfortunately in practice, the Banach module $M$ does depend on $A$ and cannot necessarily be extended to a bigger affinoid. In the Coleman--Mazur eigencurve setting, these Banach-modules are spaces of functions with a certain ``convergence radius'' which depends on $A$ and which shrinks when we increase $\Spa(A,A^+)$. 

There are at least two ways of dealing with this issue. One way is explained in~\cite{buz}. There, it is observed that the Banach-modules over the different affinoids are ``linked'' in a certain way. These links are put in an axiomatic framework and one can then prove that the glueing works. 

A second way is to make the observation that in practice even though the Banach-modules depend on an extra parameter like a radius $r$, the resulting Fredholm series and hence the spectral variety are independent of this choice. As a consequence one can then glue the spectral varieties to a ``global'' spectral variety $\mc{Z}$ over $\mc{W}$. The resulting ``global'' structure map $w\colon \mc{Z}\ra \mc{W}$ still has all properties of Theorem \ref{T.B.1.}. Next, one verifies that the coherent sheaf $\mathcal{N}$ from Proposition \ref{5.3} above is independent of the additional parameter $r$ and glues to a coherent sheaf on the glued spectral variety $\mathcal{Z}$. Once this independence is shown, everything fits into the following elegant general setup. 

\begin{defn}
An \emph{eigenvariety datum} is a tuple $\mathcal{D}=(\mathcal{Z},\mathcal{N},\T,\psi)$ where $\mathcal{Z}$ is a locally noetherian analytic adic space over $\ok$, $\mathcal{N}$ is a coherent $\mathcal{O}_{\mathcal{Z}}$-module, $\T$ a commutative $\ok$-algebra, and $\psi\colon \T \ra 
\End_{\mathcal{O}_{\mathcal{Z}}}(\mathcal{N})$ an $\ok$-algebra homomorphism.
\end{defn}

We see that from a spectral datum as above, we can build an eigenvariety datum by taking $\cZ$ to be the spectral variety and $\mathcal{N}$ the coherent sheaf from Proposition \ref{5.3}. In practice, $\mc{Z}$ will arise as a ``glued'' spectral variety over a non-affinoid weight space $\mc{W}$ as described above, and $\mc{N}$ as the corresponding glued coherent sheaf. 
Note however that in the definition of an eigenvariety datum $\mathcal{Z}$ can be even more general.

\begin{prop}\label{eigenvariety}
Given an eigenvariety datum $\mathcal{D}=(\mathcal{Z},\mathcal{N},\T,\psi)$, there is a locally noetherian analytic adic space 
$\mc{E}=\mathcal{E}(\mathcal{D})$ over $\ok$, together with a finite morphism $\pi\colon \mathcal{E} \ra \mathcal{Z}$, an $\ok$-algebra homomorphism $\phi_{\mathcal{E}}\colon \T \ra 
\mathcal{O}(\mc{E})$, and a faithful coherent \linebreak $\mc{O}_{\mc{E}}$-module $\mc{H}$. 
There is 
a canonical isomorphism $\pi_{\ast}\mc{H}\cong \mc{N}$, which is 
compatible with the actions of $\T$ on both sides (via $\phi_{\mc{E}}$ and 
$\psi$, respectively).

The space $\mc{E}$ is characterized by the following local description: for $U 
\subset \mc{Z}$ an affinoid open we have $\mc{E}_U = \Spa(\T_U, \T_U^+)$ 
where $\T_{U}$ is the $\mc{O}_{\mc{Z}}(U)$-subalgebra of 
$\End_{\mc{O}_{\mc{Z}}(U)}(\mc{N}(U))$ generated by the image of $\, \T$ and $\T_U^+$ is the integral closure of $\mc{O}^+_{\mc{Z}}(U) $ in $\T_U$. As  
$\mc{N}(U)$ is canonically a $\T_U$-module this gives $\mc{H}$ over $\mc{E}_U$.
\end{prop}
\begin{proof} For $U\subset \mc{Z}$ affinoid open, $\T_U$ is commutative and finite over $\mc{O}_{\mc{Z}}(U)$, and hence $(\T_U, \T_U^+)$ is a noetherian Tate-Huber pair. The space $\mc{E}_U$ comes with a canonical finite map $\mc{E}_U\ra U$ and $\mc{N}(U)$ is a finitely generated $\T_U$-module. By Lemma \ref{nilpotence} and flatness of rational localization for affinoid noetherian analytic adic spaces, these constructions glue together and satisfy the statements in the theorem.
\end{proof}

\begin{rem}
One shouldn't worry too much about the rings of integral elements here. In practice, it is often the case that the weight space $\mathcal{W}$ is locally of the form $\spa(A,A^{\circ})$, i.e., the rings of integral elements are the subrings of $A$ of power-bounded elements. Then every element of $\mathit{Cov}(\cZ)$ is also of the form $\spa(C,C^{\circ})$ (as the relevant morphism is finite) and as $\cE$ is finite over $\cZ$, the same is true for the eigenvariety.  
\end{rem}

\subsection{Points of eigenvarieties with pseudorigid base}
Finally we come to the most important property of eigenvarieties, which justifies their name, which is that their points correspond to systems of eigenvalues.  We refer to \cite[Chapter 2 and Section 3.7]{bellaiche} for more details regarding some of the arguments below. 

To understand the meaning of points on an eigenvariety we need to pass to a more restricted set of base spaces. We choose to work with so called pseudorigid spaces. The property of these spaces, that is important for us, is that there is a well behaved set of maximal points corresponding to maximal ideals of the rings underlying the affinoids. (For pseudorigid spaces, these maximal points are also dense.) We could work in a more general setting, namely with so called Jacobson adic spaces (see \cite{lou}), but pseudorigid spaces are the naturally occurring spaces in the context of eigenvarieties and a good setup for studying the geometry of eigenvarieties further (e.g., in order to make sense of irreducible components). We refer to the appendix for the definitions and basic facts on pseudorigid spaces. 

So let $\mathcal{W}$ be a pseudorigid space over $\ok$ and assume we have a (if need be glued) spectral variety $\mc{Z}$ and an eigenvariety datum as above with associated eigenvariety $\cE\rightarrow \mc{Z}\ra  \mathcal{W}$. In particular let us assume that there is a cover of $\mc{W}$ by open pseudorigid affinoids $W=\spa(A)=\spa(A,A^{\circ})$ that come with a fixed spectral datum $((A,A^{\circ}),M_A,\T,\phi)$ such that $\mc{Z}_W=\mc{Z}\times_{\mc{W}} W$ is the corresponding spectral variety. In the following we denote the fibre of $M_A$ above a maximal  point $w\in \Max(A)$ with residue field $L$ by $M_w:=M_A\widehat{\otimes}_A L$. 
We let $\kappa\colon \mc{E}\rightarrow \mc{W}$ denote the structure map to weight space. Let $\Max(\cE)$ be the set of maximal points of $\cE$, i.e., the set of all points $x \in \mc{E}$ such that there is an open affinoid pseudorigid neighborhood $U=\spa(A)$ of $x$ in $\cE$ such that $x$ corresponds to a maximal ideal of $A$. We want to show that the maximal points $\Max(\cE)$ of $\cE$ correspond to certain systems of eigenvalues of $\T$.

%Now let $W=\spa(A,A^+)$ be a pseudorigid space and assume with have a spectral datum $((A,A^+), M , \T, \phi)$ with associated eigenvariety $\cE\rightarrow W$.
\smallskip

For this let us first recall some notions from linear algebra. Assume that $A=L$ is a field. Recall that an element $m \in M_L$ is called a common eigenvector for $\T$ if for every $t\in \T$ there is a scalar $\lambda(t)\in L$ such that $t\cdot m = \lambda(t) m $. The map $\lambda \colon \T \rightarrow L$ is then a ring homomorphism (character). We let $M_L[\lambda]\subseteq M$ denote the set of all common eigenvectors for $\lambda$. \begin{defn}
We say that  a character $\lambda \colon \T \rightarrow L$ is a system of eigenvalues appearing in $M_L$ if $M_L[\lambda]\neq 0$. We say that $\lambda$ is \emph{$\phi$-finite} if $\lambda(\phi)\neq 0$. 
\end{defn}

Let $w\in \Max(\mc{W})$ be a maximal point with residue field $L=k(w)$. Consider the fiber $N_L:=\mathcal{N}\otimes_A L$ which is a finite-dimensional $L$ vector space. Let $T_L\subset \mathrm{End}_L(N_L)$ be the $L$-algebra generated by the image of $\T\otimes_{\mc{O}_K}L $. This is a finite algebra over $L$ and hence an artinian semi-local ring. Note that then for any of the finitely many maximal ideals $\mathfrak{m}$ of $T_L$ we have $N_L[\mathfrak{m}]\neq 0$ (cf. \cite[Section 2.5.1]{bellaiche}).

Assume $\lambda\colon \T \rightarrow L$ is a character. Then $\lambda$ is a system of eigenvalues appearing in $N_L$ if and only if $\lambda$ factors through a character $T_L\rightarrow L$. (See \cite[Theorem 2.5.9]{bellaiche}). Moreover (see \cite[Corollary 2.5.10]{bellaiche}), for any field $L'$ containing $L$, there is a natural bijection between $\Spec(T_L)(L')$ and the systems of eigenvalues of $\T$ that appear in $N_L\otimes_L L'$.

\begin{prop}
Let $w \in \Max(\mc{W})$ be a maximal point with residue field $L=k(w)$. Let $L'$ be a finite extension of $L$. Then the set of $L'$-points $x$ of $\Max(\mc{E})$ such that $\kappa(x)=w$ is in natural bijection with the $\phi$-finite systems of eigenvalues of $\T$ appearing in $M_w\otimes_{L}L'$. The bijection attaches to $x$ the system of eigenvalues $\lambda_x \colon \T \rightarrow L', t\mapsto \phi_{\mc{E}}(t)(x)$.  
\end{prop}

%\begin{prop}\label{points1}
%There is a natural bijection between $\phi$-finite $L$-valued systems of eigenvalues and the $L$-points of $\Max(\cE)$. 
%\end{prop}
\begin{proof}
Without loss of generality $\mc{W}$ is affinoid. Also $\cE$ is covered by the $\cE_U$, for $U \in \mathit{Cov}(\cZ)$ and we can work locally on $\cE$. So we fix $U\in \mathit{Cov}(\cZ)$ and let $V=\spa(A)\subseteq \mathcal{W}$ be its image in $\mathcal{W}$ and we may assume $w \in V$. 
This point corresponds to a surjection $A\rightarrow L$. Note that an $L'$-point $x\in \Max(\mc{E}_U)$ gives a map $x\colon \T_U \rightarrow L'$. Furthermore, by construction, a system of eigenvalues of $\T$ appearing in $M_w\otimes_{L}L'$  is $\phi$-finite if and only if it appears in $(\mc{N}(U')\otimes_A L)\otimes L'$ for some $U' \in \mathit{Cov}(\cZ)$. Let $N_w = \mathcal{N}(U)\otimes_A L$. 

To prove the proposition it now suffices to construct a bijection between the maximal $L'$-points in $\mc{E}_U$ above $w$ and the systems of eigenvalues appearing in  $N_w\otimes_{L}L'$. But by Lemma \ref{nilpotence} above, the kernel of the map $\T_U\otimes_A L \rightarrow T_L$ is nilpotent and so a maximal ideal in $\mc{E}_U$ corresponds to a maximal ideal in $T_L$. The result now follows from the linear algebra discussion prior to this proposition. 
\end{proof}

\subsection{Rigidity results}
We end our exposition of the eigenvariety machine by discussing some rigidity results,  following \cite[Section 3]{JN2} closely. For simplicity we keep the assumption that the base $\mc{W}$ (or $\mc{Z}$) is pseudorigid over $\ok$. 
Below we use the following notation: If $X$ is a pseudorigid space over $\ok$, $x\in \Max(X)$ is a maximal point and $\mc{N}$ is a coherent $\mc{O}_{X}$-module, then we write $\mc{N}(x)$ for the fiber of $\mc{N}$ at $x$, which is a vector space over the residue field $k(x)$.

Let us first describe the behavior of the eigenvariety construction under arbitrary base change.
\begin{prop}\label{basechange}
Let $\mc{D}=(\mc{Z},\mc{N},\T,\psi)$ be an eigenvariety datum with 
eigenvariety $\mc{E}$. Let $f \colon \mc{Z}^{\prime} \ra \mc{Z}$ be a map of 
pseudorigid spaces over $\ok$. Form the eigenvariety datum 
$$\mc{D}^{\prime}:=(\mc{Z}^{\prime}, \mc{N}^{\prime}:=f^{\ast}\mc{N}, \T, \psi^{\prime}), $$
where $\psi^{\prime}$ is the composition of $\psi$ with the natural map $\End_{\mc{O}_{\mc{Z}}}(\mc{N}) \ra \End_{\mc{O}_{\mc{Z}^{\prime}}}(\mc{N}^{\prime})$. Let $\mc{E}^{\prime}$ be the eigenvariety attached to $\mc{D}^{\prime}$. Then there is a natural map $\mc{E}^{\prime} \ra \mc{E}$ over $\mc{Z}$, and the induced map $\mc{E}^{\prime} \ra \mc{E}\times_{\mc{Z}}\mc{Z}^{\prime}$ induces an isomorphism $(\mc{E}^{\prime})^{red} \ra (\mc{E}\times_{\mc{Z}}\mc{Z}^{\prime})^{red}$. If $f$ is flat, then $\mc{E}^{\prime} \ra \mc{E}\times_{\mc{Z}}\mc{Z}^{\prime}$ is an isomorphism.
\end{prop}

\begin{proof}
By the local nature of the construction of eigenvarieties in Proposition \ref{eigenvariety}, the assertion is local both on $\mc{Z}^{\prime}$ and $\mc{Z}$, so we may assume that they are both affinoid pseudorigid spaces over $\ok$. Then the proposition follows directly from Lemma \ref{nilpotence}.
\end{proof}

Applying this to the case, where $f$ is the closed immersion $\mc{Z}^{\prime}=z \hookrightarrow \mc{Z}$ of a maximal point $z\in \Max(\mc{Z})$, we can also slightly rephrase the results from the previous section as follows. 
\begin{corollary}\label{points}
Let $\mc{D}=(\mc{Z},\mc{N},\T,\psi)$ be an eigenvariety datum with eigenvariety \linebreak $\pi \colon \mc{E} \ra \mc{Z}$. Fix $z\in \Max(\mc{Z})$. Then the set $\pi^{-1}(z)\sub \Max(\mc{E})$ is in natural bijection with the systems of Hecke eigenvalues appearing in the fibre $\mc{N}(z)$, i.e., the maximal ideals lying above the kernel of the natural map $\T\otimes_{\ok}k(z) \ra \End_{k(z)}(\mc{N}(z))$.
\end{corollary}

Let us also mention the following reconstruction result for eigenvariety data.

\begin{prop}\label{reconstruction}
Let $\pi \colon \mc{E} \ra \mc{Z}$ be a finite morphism of pseudorigid spaces and let $\T$ be an $\ok$-algebra. Assume that there is a ring homomorphism $\phi \colon \T \ra \mc{O}(\mc{E})$ such that $\mc{O}_{\mc{E}}$ is generated by the image of $\phi$ over $\mc{O}_{\mc{Z}}$, and assume that $\mc{H}$ is a faithful coherent $\mc{O}_{\mc{E}}$-module. We may form an eigenvariety datum
$$ \mc{D}=(\mc{Z},\mc{N},\T,\psi), $$
where $\mc{N}=\pi_{\ast}\mc{H}$ and $\psi$ is the composition $\T \overset{\phi}{\ra} \mc{O}(\mc{E}) \ra \End_{\mc{O}_{\mc{Z}}}(\mc{N})$. Then $\pi \colon \mc{E} \ra \mc{Z}$ is the eigenvariety attached to $\mc{D}$, and $\phi=\phi_{\mc{E}}$.
\end{prop}
\begin{proof} See \cite[3.1.6]{JN2}.
\end{proof}

\begin{exer} Show that any Zariski closed subset of an eigenvariety is naturally the eigenvariety of an eigenvariety datum.
\end{exer}

We end with a very general result that makes it possible to construct morphisms between eigenvarieties by interpolating maps that are merely defined on a dense set of points. 

\begin{thm}\label{optimalint}
Let $\mc{D}_{i}=(\mc{Z}_{i},\mc{N}_{i},\T_{i},\psi_{i})$, for $i=1,2$, be eigenvariety data with corresponding eigenvarieties $\pi_i\colon \mc{E}_i \rightarrow \mc{Z}_i$. Assume that we have the following additional data:
\begin{itemize}
\item A morphism $j \colon \mc{Z}_{1} \ra \mc{Z}_{2}$ of adic spaces;

\item An $\ok$-algebra homomorphism $\sigma \colon \T_{2} \ra \T_{1}$;

\item A subset $\mc{E}^{cl}\sub \Max(\mc{E}_{1})$ such that the 
$\T_{2}$-eigensystem of $x$ (i.e., the $\T_{1}$-eigensystem of $x$ 
composed with $\sigma$) appears in $\mc{N}_{2}(j(\pi_{1}(x)))$ for all $x\in 
\mc{E}^{cl}$.
\end{itemize}
Let $\ol{\mc{E}}$ denote the Zariski closure of $\mc{E}^{cl}$ in $\mc{E}_{1}$, with 
its induced reduced structure. Then there is a canonical morphism 
$$ i \colon \ol{\mc{E}} \ra \mc{E}_{2} $$
lying over $j \colon \mc{Z}_{1} \ra \mc{Z}_{2}$ such that $\phi_{\ol{\mc{E}}}\circ \sigma = i^{\ast}\circ \phi_{\mc{E}_{2}}$. The morphism $i$ inherits the following properties from $j$:

\begin{itemize}
\item If $j$ is (partially) proper (resp. finite), then $i$ is (partially) proper (resp. finite);
 
\item If $j$ is a closed immersion and $\sigma$ is a surjection, then $i$ is a closed immersion.
\end{itemize}
\end{thm}
\begin{proof} See \cite[Theorem 3.2.1]{JN2}.
\end{proof}
In practice this interpolation result (or similar versions of it, cf. \cite[Prop.3.5]{chenevier}, \cite[Prop. 7.2.8]{BC}, \cite[Thm. 5.1.6]{hansen}) is a very helpful tool  to compare eigenvarieites for different eigenvariety data.  
For example one can use it to show ``independence of the choice of compact operator''. Slightly more precisely, let us assume a situation where $\mc{Z}_1$ is the spectral variety of a compact operator acting on spaces of $p$-adic automorphic forms. In such situations there often is a flexibility in the choice of the compact operator (leading to a second spectral variety $\mc{Z}_2$ and a second eigenvariety datum.). The interpolation result above can be used to show that the resulting reduced eigenvarieties are isomorphic, as long as the compact operators share the same finite-slope parts (cf. e.g. \cite[Lemma 3.4.2]{JN2}).

The set $\mathcal{E}^{cl}$ in the theorem above typically arises from a set of \emph{classical} automorphic forms (hence the notation). Now there might be different eigenvariety data that capture the same set of systems of Hecke eigenvalues, e.g. there might be more than one way of constructing $p$-adic automorphic forms that ``see'' these classical forms. The interpolation result then produces maps between the resulting eigenvarieties. 
A further particularly interesting application is in the context of Langlands functoriality. There the interpolation theorem can be applied to construct maps between eigenvarieties that interpolate classical Langlands functoriality $p$-adically thereby comparing $p$-adic automorphic forms for different reductive groups (see e.g., \cite{chenevier, hansen, JN2, ludwig}).

\appendix
\section{Pseudorigid spaces}
Let $K$ be a complete discretely valued field with ring of integers $\cO_K$, uniformizer $\pi$ and residue field $k$. In this appendix we discuss \emph{pseudorigid spaces} over $\mc{O}_K$. They are a generalization of rigid analytic spaces that include the analytic loci of formal schemes, formally of finite type over $\mc{O}_K$. In the corresponding class of Huber rings one allows an additional formal power series variables. Base spaces of eigenvarieties occurring in practice are pseudorigid spaces, which is our motivation for discussing them. Besides a brief introduction to the concept we explain that pseudorigid spaces have a well-behaved set of \emph{maximal points}. The main references for the theory of pseudorigid spaces are \cite{JN1}, \cite{JN2}, \cite{lou} and \cite{abbes}.  
Further topics that are important in the context of eigenvarieties are for example the theory of normalization of pseudorigid spaces and their dimension theory. We will not discuss these topics here but instead refer the interested reader to the references mentioned above.

In the following we will frequently deal with affinoid adic spaces where the ring of integral elements $A^+$ is given by the subring of power-bounded elements $A^\circ$. In this case we shorten the notation and write $\spa(A)$ for $\spa(A,A^\circ)$. 

\begin{defn}
Let $A$ be a complete Tate $\cO_K$-algebra. We say that $A$ is a \emph{pseudoaffinoid} $\cO_K$-algebra, if $A$ has a ring of definition $A_0$ which is formally of finite type over $\cO_K$, i.e., $A_0$ has a radical ideal of definition $I$, such that $A_0/I$ is a finitely generated $k$-algebra. A homomorphism $A\rightarrow B$ of pseudoaffinoid $\cO_K$-algebras is a continuous homomorphism $A\rightarrow B$. 
\end{defn}

\begin{rem} 
\begin{itemize}
\item Note that the definition makes sense, as  $\cO_K\rightarrow A_0$ is continuous, so then $\pi^n\subset I$ for some $n$, and as $I$ is radical, $\pi$ itself is contained in $I$. 
%\item In \cite{JN2} a pseudoaffinoid $\cO_K$-algebra is called a Tate ring formally of finite type. We prefer the term pseudoaffinoid from \cite{Lou}.
\item The topology on a ring of definition $A_0$ for a Tate $\cO_K$-algebra $A$ is the $(\varpi)$-adic topology for some pseudo-uniformizer $\varpi \in A$. We can define a map $\cO_K\llbracket X\rrbracket\rightarrow A_0$ by sending $X$ to $\varpi$ and then $A_0$ is formally of finite type over $\cO_K$ if and only if it is of strictly topologically finite type over $\cO_K\llbracket X\rrbracket$, i.e., if and only if it is topologically isomorphic to a quotient of $\cO_K\llbracket X\rrbracket\langle Y_1,\dots, Y_n\rangle$ for some $n$. In particular we see that $A_0$ is noetherian.
\item Pseudoaffinoid $\cO_K$-algebras are Jacobson (\cite[A.1]{JN1}) and excellent (\cite[2.2.3]{JN2}) which is the reason for some of their good properties. 
\end{itemize}
\end{rem}
 
\begin{prop}\label{structurepr}
Let $A$ be a pseudoaffinoid $\cO_K$-algebra.
\begin{enumerate}
\item Let $B$ be a complete Huber ring with a morphism $A\rightarrow B$ which is topologically of finite type. Then $B$ is also pseudoaffinoid. 
\item Every homomorphism of pseudoaffinoid $\cO_K$-algebras $A \rightarrow B$ is topologically of finite type.
\item If $U\subset X=\spa(A,A^\circ)$ is a rational subset, then $\cO_X(U)^+=\cO_X(U)^{\circ}$.
\item More generally if $U\subset X$ is affinoid open, then  $\cO_X(U)^+=\cO_X(U)^{\circ}$.
\item Suppose $f\colon A\rightarrow B$ is a morphism of pseudoaffinoid $\cO_K$-algebras such that the induced morphism $\spa(B)\rightarrow \spa(A)$ is an open immersion. Let $\mathfrak{m}$ be a maximal ideal of $B$, then $f^{-1}(\mathfrak{m})$ is a maximal ideal of $A$ and the natural map $A_{f^{-1}(\mf{m})}\rightarrow B_{\mf{m}}$ induces an isomorphism on completions. 
\end{enumerate}
\end{prop}

\begin{proof}
The first part is left as an exercise. For the second part see  \cite[Corollary A.14]{JN1}. The third  part follows from Lemma \ref{ic1} below upon noting that by construction $\cO_X(U)^+$ contains a ring of definition of $\cO_X(U)$. For part (4) note that $U$ has a finite cover by Tate rational subdomains $(U_i)_{i \in I}$ of $X$. Since the maps $\cO(U)\rightarrow \cO(U_i)$ are bounded (the $U_i$ are also rational subdomains of $U$), the strict embedding \[\cO(U)\hookrightarrow \prod_{i\in I}\cO(U_i)\] induces an embedding  \[\cO(U)^\circ \hookrightarrow \cO(U)\cap\prod_{i\in I}\cO(U_i)^\circ\] but the right hand side equals \[\cO(U)\cap\prod_{i\in I}\cO^+(U_i)=\cO^+(U)\] by part (3), so we are done because by definition $\cO^+(U) \subseteq \cO(U)^\circ$ which implies that we have equality. 
Finally for part (5) we refer to \cite[Proposition 2.2.3]{JN2}
\end{proof}

 \begin{lem}\label{ic1} Let $A$ be a complete Tate ring with a Noetherian ring of definition $A_0$. Let $B\subseteq A$ be a ring of definition containing $A_0$. Then $B$ is a finitely generated $A_0$-module, hence Noetherian and integral over $A_0$. Moreover $A^{\circ}$ is the integral closure of $A_0$ in $A$. 
 \end{lem}
 \begin{proof} Pick a pseudo-uniformizer $\varpi \in A$ contained in $A_0$. Since $B$ is bounded, we have that $B\subseteq \varpi^{-N}A_0$ for some $N$. Hence $B$ is an $A_0$-submodule of the cyclic $A_0$-module $\varpi^{-N}A_0$. The first claim now follows. For the last assertion, first note that the integral closure of $A_0$ is contained in $A^\circ$. Since $\Acirc$ is the union of all open and bounded subrings, and any two open bounded subrings are contained in a third, the assertion follows from the first part.  
 \end{proof}

In particular, we see that every rational subset of $\spa(A, A^{\circ})$, where $A$ is a pseudo-affinoid $\cO_K$-algebra, is of the form $\spa(B, B^{\circ})$ for a pseudoaffinoid \linebreak $\cO_K$-algebra $B$.
 
 \begin{defn} Let $X$ be an adic space over $\spa(\cO_K)$. Then $X$ is called a \emph{pseudorigid space over} $\cO_K$ if $X$ is locally of the form $\Spa(A,A^{\circ})$ for a pseudoaffinoid $\cO_K$-algebra~$A$. 
 \end{defn}

\begin{ex}
Let $\ms{X}$ be a formal scheme, formally of finite type over $\cO_K$, which we regard here as an adic space in the sense of \cite{hu2,hu1}. Then the set of analytic points $X=\ms{X}_a$ is a pseudorigid space over $\cO_K$. For example, if $\ms{X}=\mathrm{Spf} (\Z_p\llbracket T\rrbracket)$, then $X$ is a quasicompact analytic adic space, covered by two affinoids, namely $\spa(\Q_p \langle T/p \rangle)$ and $\spa (\Z_p\llbracket T\rrbracket, \Z_p\llbracket T\rrbracket) \left( \frac{p,T}{T} \right)$ (compare Example \ref{tatenofield}). 
\end{ex}

Recall that $\spa(\cO_K, \cO_K)$ consists of two points, the closed point $s=\spa(k,k)$ and the generic point $\eta=\spa (K, \cO_K)$. For a pseudorigid space $X$, the fibre products $X_s=X\times_{\spa(\cO_K)} \spa (k)$ and $X_{\eta}=X\times_{\spa (\cO_K)} \spa (K)$ exist and we call these the \emph{special}, resp. \emph{generic} fibre of $X$. One can show that $X_s$ is the Zariski closed adic subspace of $X$ defined by the ideal sheaf $\pi \cO_X$ and $X_{\eta}$ is its open complement.
%\todo{reference for fibre product}
\begin{prop}
The generic fibre $X_{\eta}$ of a pseudorigid space $X$ over $\cO_K$ is a $K$-rigid space. If $X=\spa(A)$ with $A$ pseudoaffinoid, then $X_s$ admits the structure of a rigid space over the field ~$k((T))$.
\end{prop}

\begin{proof}
For the first part we can work locally on $X_{\eta}$ and assume that $X_{\eta}=\spa(B)$ where $B$ is a pseudoaffinoid $\cO_K$-algebra and $\pi \in B^{\times}$. Then the canonical map $K \rightarrow B$ is topologically of finite type by Proposition \ref{structurepr}. For the claim about the special fiber observe that if $B$ is pseudoaffinoid with $\pi B=0$, we may choose a topologically nilpotent unit $\varpi \in B$ and look at the continuous $\cO_K$-homomorphism $k((T)) \rightarrow B$ that sends $T$ to $\varpi$. Once again, this must be topologically of finite type, which yields the claim.
\end{proof}

So now that we have a bit of an idea of how pseudorigid spaces look like, let us discuss the fact that they have a well-behaved set of points defined by maximal ideals in pseudoaffinoid $\cO_K$-algebras. 

%\begin{defn}
%A noetherian adic ring $B$ is called a \emph{valuative order} if it is an integral domain which is local of Krull dimension $1$, and has no $J$-torsion, where $J$ is an ideal of definition (this is independent of the choice of ideal of definition).
%\end{defn}

\begin{lem}\label{ps4}
Let $A$ be a pseudoaffinoid $\cO_K$-algebra, and let $\mf{m}\subset A$ be a maximal ideal. Then there exists a unique $x\in \Spa(A)$ with $\mathrm{supp}(x)=\mf{m}$. This gives a canonical embedding of the spectrum $\mathrm{Max}(A)$ of maximal ideals into $\Spa(A)$, compatible with the morphism $\mathrm{supp}\colon \Spa(A) \rightarrow \Spec(A)$. 
\end{lem} 

\begin{proof} 
The locus of $x\in \Spa(A)$ with $\mathrm{supp}(x)=\mf{m}$ is equal to 
$\Spa(A/\mf{m})$ and one shows that this is a singleton, see \cite[Lemma 2.2.5]{JN2} for the argument.
\begin{comment}
The locus of $x\in \Spa(A)$ with $\mathrm{supp}(x)=\mf{m}$ is equal to 
$\Spa(A/\mf{m})$ %(note that $A/\mf{m}$ is finite over $A$, and we are using Lemma \ref{powerbounded}), 
so we have to prove 
that this is a singleton. Let $A_{0}$ be a ring of definition formally of 
finite type over $\ok$, and let $\varpi \in A$ be a pseudo-uniformizer, 
which we assume lies in $A_{0}$. Let $\mf{p}=A_{0}\cap \mf{m}$. The quotient 
ring $A/\mf{m}$ carries the topology where $A_{0}/\mf{p}$ is open and carries 
the $\ol{\vp}$-adic topology, where $\ol{\vp}$ is the reduction of $\vp$. Since 
$\mf{p}$ is a closed point in $\Spec(A_{0})\setminus \{\vp=0\}=\Spec(A)$,
 \cite[Proposition 1.11.8]{abbes} implies that $A_0/\mf{p}$ is a valuative order.
As a result, $A/\mf{m}=(A_{0}/\mf{p})[1/\ol{\vp}]$ is a complete discretely valued field, 
with valuation ring the integral closure of $A_{0}/\mf{p}$, and $A_{0}/\mf{p}$ 
is open in the valuation topology. It follows that the valuation topology is 
equal to the quotient topology coming from $A$, and hence that $\Spa( A/\mf{m})$ 
is a singleton, as desired. This proves the first statement, and the remaining 
statements are clear from the definition of the support map.
\end{comment}
\end{proof}

\begin{prop}\label{ps5}
Let $A$ and $B$ be pseudoaffinoid $\ok$-algebras, and let $f \colon A \ra B$ be a continuous $\ok$-homomorphism. Then the preimages of maximal ideals are maximal ideals. Hence the induced map $f\colon  \Spa(B) \ra \Spa(A)$ maps $\Max(B)$ into $\Max(A)$. Moreover, $\Max(A)$ is dense in $\Spa(A)$.
\end{prop}

\begin{proof}
The first part follows from the fact that $f$ is of topologically finite type (cf.\ \cite[A.13,A.14]{JN1}). For density claim, note that if $U\sub X=\Spa(A)$ is a rational subset, then the first part implies that $\Max(\cO_{X}(U))=U\cap \Max(A)$. It follows that $U\cap \Max(A)$ is non-empty if $U$ is, and hence that $\Max(A)$ is dense.
\end{proof}

\begin{defn}\label{ps6}
Let $X$ be a pseudorigid space. We say that a point $x\in X$ is \emph{maximal} if there exists an open affinoid pseudorigid neighbourhood $U=\Spa(A)$ of $x$ such that $x\in \Max(A)\sub U$. We denote the set of all maximal points of $X$ by $\Max(X)$.
\end{defn}

Note that this is a local property, for if $x\in \Max(X)$ and $V=\Spa(B)\sub X$ is an open affinoid pseudorigid neighbourhood, then by Proposition \ref{ps5}, $x\in \Max(B)$. Moreover, it also follows that $\Max(X)$ is dense in $X$, that $\Max(X)=\Max(A)$ if \linebreak  $X=\Spa(A)$ is affinoid pseudorigid, and that any morphism $f \colon X \ra Y$ of pseudorigid spaces maps $\Max(X)$ into $\Max(Y)$.

\newpage

\bibliography{ASII}

\begin{thebibliography}{10}

\bibitem{abbes}
Ahmed {Abbes}.
\newblock {\em {\'El\'ements de g\'eom\'etrie rigide. Vol. I. Construction et
  \'etude g\'eom\'etrique des espaces rigides.}}, volume 286.
\newblock Basel: Birkh\"auser, 2011.

\bibitem{aip}
Fabrizio {Andreatta}, Adrian {Iovita}, and Vincent {Pilloni}.
\newblock {Le halo spectral.}
\newblock {\em {Ann. Sci. \'Ec. Norm. Sup\'er. (4)}}, 51(3):603--655, 2018.

\bibitem{as}
Avner {Ash} and Glenn {Stevens}.
\newblock $p$-adic deformations of arithmetic cohomology.
\newblock Preprint,
  \url{http://math.bu.edu/people/ghs/preprints/Ash-Stevens-02-08.pdf}.

\bibitem{bellaiche}
Jo{\"e}l Bella{\"{\i}}che.
\newblock {\em The eigenbook. {Eigenvarieties}, families of {Galois}
  representations, {{\(p\)}}-adic {{\(L\)}}-functions}.
\newblock Pathw. Math. Cham: Birkh{\"a}user, 2021.

\bibitem{BC}
Jo{\"e}l Bella{\"{\i}}che and Ga{\"e}tan Chenevier.
\newblock {\em Families of {Galois} representations and {Selmer} groups},
  volume 324 of {\em Ast{\'e}risque}.
\newblock Paris: Soci{\'e}t{\'e} Math{\'e}matique de France, 2009.

\bibitem{bergdall}
John {Bergdall}.
\newblock {An introduction to Huber rings and valuation spectra}.
\newblock In {\em {this volume}}. 2023.

\bibitem{BGR}
S.~Bosch, U.~G{\"u}ntzer, and R.~Remmert.
\newblock {\em Non-{A}rchimedean analysis}, volume 261 of {\em Grundlehren der
  Mathematischen Wissenschaften [Fundamental Principles of Mathematical
  Sciences]}.
\newblock Springer-Verlag, Berlin, 1984.
\newblock A systematic approach to rigid analytic geometry.

\bibitem{bosch}
Siegfried Bosch.
\newblock {\em Lectures on formal and rigid geometry}, volume 2105 of {\em
  Lect. Notes Math.}
\newblock Cham: Springer, 2014.

\bibitem{buz}
Kevin Buzzard.
\newblock Eigenvarieties.
\newblock In {\em {$L$}-functions and {G}alois representations}, volume 320 of
  {\em London Math. Soc. Lecture Note Ser.}, pages 59--120. Cambridge Univ.
  Press, Cambridge, 2007.

\bibitem{chenevier}
Ga{\"e}tan Chenevier.
\newblock A {{\(p\)}}-adic {Jacquet}-{Langlands} correspondence.
\newblock {\em Duke Math. J.}, 126(1):161--194, 2005.

\bibitem{cm}
R.~{Coleman} and B.~{Mazur}.
\newblock {The eigencurve.}
\newblock In {\em {Galois representations in arithmetic algebraic geometry.
  Proceedings of the symposium, Durham, UK, July 9--18, 1996}}, pages 1--113.
  Cambridge: Cambridge University Press, 1998.

\bibitem{col}
Robert~F. {Coleman}.
\newblock {\(p\)-adic Banach spaces and families of modular forms.}
\newblock {\em {Invent. Math.}}

\bibitem{gouv}
Fernando~Q. Gouv\^{e}a.
\newblock {\em {$p$}-adic numbers}.
\newblock Universitext. Springer-Verlag, Berlin, 1993.
\newblock An introduction.

\bibitem{hansen}
David Hansen and James Newton.
\newblock Universal eigenvarieties, trianguline {Galois} representations, and
  {{\(p\)}}-adic {Langlands} functoriality.
\newblock {\em J. Reine Angew. Math.}, 730:1--64, 2017.

\bibitem{hu2}
R.~Huber.
\newblock A generalization of formal schemes and rigid analytic varieties.
\newblock {\em Math. Z.}, 217(4):513--551, 1994.

\bibitem{hu1}
Roland Huber.
\newblock {\em \'{E}tale cohomology of rigid analytic varieties and adic
  spaces}.
\newblock Aspects of Mathematics, E30. Friedr. Vieweg \& Sohn, Braunschweig,
  1996.

\bibitem{johansson}
Christian {Johansson}.
\newblock {Coherent sheaves and examples}.
\newblock In {\em {this volume}}. 2023.

\bibitem{JN1}
Christian {Johansson} and James {Newton}.
\newblock {Extended eigenvarieties for overconvergent cohomology.}
\newblock {\em {Algebra Number Theory}}, 13(1):93--158, 2019.

\bibitem{JN2}
Christian {Johansson} and James {Newton}.
\newblock {Irreducible components of extended eigenvarieties and interpolating
  Langlands functoriality.}
\newblock {\em {Math. Res. Lett.}}, 26(1):159--201, 2019.

\bibitem{lou}
João N.~P. Lourenço.
\newblock {The Riemannian Hebbarkeitssätze for pseudorigid spaces.}
\newblock \url{https://arxiv.org/abs/1711.06903}.

\bibitem{ludwig}
Judith Ludwig.
\newblock A {{\(p\)}}-adic {Labesse}-{Langlands} transfer.
\newblock {\em Manuscr. Math.}, 154(1-2):23--57, 2017.

\bibitem{mor}
Sophie Morel.
\newblock Lecture notes on {A}dic {S}paces.
\newblock 2018.
\newblock \url{https://web.math.princeton.edu/~smorel/adic_notes.pdf}.

\bibitem{neukirch}
J{\"u}rgen Neukirch.
\newblock {\em Algebraic number theory. {Transl}. from the {German} by
  {Norbert} {Schappacher}}, volume 322 of {\em Grundlehren Math. Wiss.}
\newblock Berlin: Springer, 1999.

\bibitem{newton}
James {Newton}.
\newblock {Construction of eigenvarieties}.
\newblock In {\em {this volume}}. 2023.

\bibitem{schneider}
Peter Schneider.
\newblock {\em Non-archimedean functional analysis}.
\newblock Springer Monographs in Mathematics. 2002.

\bibitem{berk}
Peter {Scholze} and Jared {Weinstein}.
\newblock {\em {Berkeley lectures on $p$-adic geometry.}}, volume 207.
\newblock Princeton, NJ: Princeton University Press, 2020.

\bibitem{serre}
Jean-Pierre {Serre}.
\newblock {Endomorphismes compl\`etement continus des espaces de Banach
  \(p\)-adiques.}
\newblock {\em {Publ. Math., Inst. Hautes \'Etud. Sci.}}

\bibitem{w}
Torsten Wedhorn.
\newblock Lecture notes on {A}dic {S}paces.
\newblock 2011.
\newblock \url{https://arxiv.org/abs/1910.05934}.

\end{thebibliography}
\bibliographystyle{plain}
\end{document}